\documentclass[10pt]{amsart}
\usepackage{fullpage,url,amssymb,enumerate,colonequals}
\usepackage{mathrsfs}
\usepackage[section]{placeins}
\IfFileExists{MnSymbol.sty}{\usepackage{MnSymbol}}{} 
\IfFileExists{extarrows.sty}{\usepackage{extarrows}}{} 
\usepackage{lscape}
\usepackage[all,cmtip]{xy}
\usepackage{graphicx} 
\usepackage{todonotes}
\IfFileExists{algpseudocodex.sty}{\usepackage[indLines=true, noEnd=false]{algpseudocodex}}{} 
\IfFileExists{ot2enc.def}{\usepackage[OT2,T1]{fontenc}}{\usepackage[T1]{fontenc}} 
\usepackage{xcolor} 
\usepackage[
        colorlinks, citecolor=darkgreen,
        backref,
        pdfauthor={Nuno Freitas},
]{hyperref}
\usepackage{comment}
\usepackage{multirow}

\numberwithin{equation}{section}

\newtheorem{lemma}[equation]{Lemma}
\newtheorem{theorem}[equation]{Theorem}
\newtheorem{proposition}[equation]{Proposition}

\newtheorem{corollary}[equation]{Corollary}

\newtheorem{claim*}{Claim}

\newtheorem{defn}[equation]{Definition}

\theoremstyle{definition}
\newtheorem{remark}[equation]{Remark}

\definecolor{darkgreen}{rgb}{0,0.5,0}
\definecolor{rem}{rgb}{0.8,0,0}
\definecolor{new}{rgb}{0.3,0.1,0.9}
\definecolor{reply}{rgb}{0,0,0.8}
\definecolor{gray}{gray}{0.7}

\renewcommand{\det}{\text{det}}

\newcommand{\F}{\mathbb{F}}
\newcommand{\PP}{\mathbb{P}}
\newcommand{\Q}{\mathbb{Q}}

\newcommand{\Z}{\mathbb{Z}}

\newcommand{\Gal}{\text{Gal}}

\newcommand{\GL}{\text{GL}}
\newcommand{\SL}{\text{SL}}
\newcommand{\SD}{\operatorname{SD}}

\newcommand{\Id}{\operatorname{Id}}
\newcommand{\Aut}{\operatorname{Aut}}
\newcommand{\Spec}{\operatorname{Spec}}

\renewcommand{\ker}{\text{ker}\:}

\DeclareMathOperator{\Tr}{\text{Tr}}

\DeclareMathOperator{\Hom}{Hom}

\DeclareMathOperator{\ord}{ord}

\DeclareMathOperator{\Frac}{Frac}

\DeclareMathOperator{\CM}{CM}

\DeclareMathOperator{\Image}{\operatorname{Im}}
\DeclareMathOperator{\Fil}{\operatorname{Fil}}

\providecommand{\cl}[1]{\mathcal{#1}}
\providecommand{\dirlim}[1]{\varinjlim\limits_{#1}}

\providecommand{\tbt}[4]{\begin{pmatrix} #1 & #2\\#3&#4\\ \end{pmatrix}}

\providecommand{\p}{\mathfrak{p}}

\providecommand{\barra}[1]{\overline{#1}}
\providecommand{\tenso}{\otimes}

\providecommand{\st}{~\vert~}
\providecommand{\Spec}{\operatorname{Spec}}
\providecommand{\Q}{\mathbb Q}
\providecommand{\Z}{\mathbb Z}

\providecommand{\ma}{\mathbb}

\newcommand{\fpgm}{\text{filtered}~(\varphi,\Gal(K/\Q_p))-\text{module}}
\newcommand{\fpgmt}{\text{filtered}~(\varphi,\Gal(K/\Q_2))-\text{module}}
\newcommand{\upgmt}{\text{unfiltered}~(\varphi,\Gal(K/\Q_2))-\text{module}}

\title{$2$-adic representations of elliptic curves over $\Q_2$}

\date{\today}

\author{Jose Castro-Moreno}
\address{Instituto de Ciencias Matem\'aticas (ICMAT),
          Nicol\'as Cabrera 13-15
         28049 Madrid, Spain}
\email{joseantonio.castro@icmat.es}

\author{Nuno Freitas}
\address{Instituto de Ciencias Matem\'aticas (ICMAT),
          Nicol\'as Cabrera 13-15
         28049 Madrid, Spain}
\email{nuno.freitas@icmat.es}

\thanks{Castro-Moreno's PhD thesis project was supported by a fellowship from the ``La Caixa'' Foundation (ID 100010434). The fellowship code is LCF/BQ/DR24/12080026.}
\thanks{Freitas was partly supported by MICINN (Spain) grant PID 2022-136944NB-I00.}

\keywords{Elliptic Curves, $p$-adic representations, Cristalline representations, Fontaine functors}
\subjclass[2020]{11F80, 11G07, 11S20}

\begin{document}

\maketitle

\begin{abstract}
    We provide a complete classification of the rational $2$-adic Tate modules $V_{2}(E)$ of elliptic curves $E/\Q_2$ with potentially good reduction. The classification is in terms of the $\fpgmt$ associated to $V_2(E)$ via the Fontaine functor $D_{pcris}^{*}$.
\end{abstract}

\section{Introduction}
Let $\ell$ and $p$ be primes. Let $K/\Q_p$ be a finite extension, $E/K$ an elliptic curve, and $T_{\ell}(E)$ its $\ell$-adic Tate module. The action of $G_K$ on $V_{\ell}(E):=T_{\ell}(E)\otimes_{\Z_{\ell}}\Q_{\ell}$ gives rise to a representation $\rho: G_{K}\rightarrow \GL_2(\Q_{\ell})$, after a choice of basis. When $\ell\neq p$, these representations have been extensively studied and are well understood. When $\ell=p$ and $E/K$ has potentially multiplicative reduction, a quadratic twist reduces the problem to a Tate curve whose $p$-adic representation has been studied in detail. Much less is known in the case $\ell=p$ when $E/K$ has potentially good reduction. Indeed, a classification of such representations exists only for $K=\Q_p$ and $p\geq 5$ by the work of Volkov~\cite{Volkov2001} and for $p=3$ due to Bosco~\cite{bosco20233adicrepresentationsarisingelliptic}. The goal of this paper is to treat the case~$p=2$, and hence  complete the classification of $p$-adic representations of elliptic curves over $\Q_p$.

When $\ell\neq p$, the Néron--Ogg--Shafarevich criterion states that an elliptic curve $E/K$ has good reduction if and only if the action of $G_K$ on $V_\ell(E)$ is unramified. When $\ell=p$, this is no longer true, and the notion of being unramified must be replaced by that of being crystalline, yielding the analogous theorem: an elliptic curve $E/K$ has good reduction if and only if $V_p(E)$ is crystalline. These facts are particularly useful, since studying the $\ell$-adic representations of $E$ provides information about its geometry that can then be used to study its $p$-adic representation. The present paper illustrates this interplay. For $p=2$, unlike the cases $p=3$ and $p=5$, the use of $\ell$-adic inertial types of elliptic curves $E/\Q_2$ is essential for identifying the relevant wildly ramified extensions over which these curves acquire good reduction. For $p=5$, tameness avoids this difficulty, while for $p=3$ the extensions are small enough to be found using the \href{https://www.lmfdb.org/padicField/}{LMFDB}. For $p=2$, however, the semistability defect can be as large as $e=24$, with the corresponding extension having a Galois closure of degree $48$ and Galois group $\GL_2(\F_3)$. The greater complexity of $\Gal(K/\Q_2)$ also makes the semilinear objects involved more complicated, affecting both the computations and the theoretical arguments.

The classification uses the anti-equivalence of categories between potentially crystalline representations and filtered $(\varphi,\Gal(\overline{\Q_2}/\Q_2))$-modules provided by the contravariant Fontaine functor $D_{pcris}^{*}$. These objects admit descriptions as $\fpgmt$s with $K/\Q_2$ a finite extension. Using this description, we classify, up to isomorphism, all modules that occur as $D_{pcris}^{*}(V_2(E))$ for $E/\Q_2$. We describe them as $\fpgmt$s, where $K/\Q_2$ is a finite Galois extension over which $E$ has good reduction. Our main theorem is as follows; the objects appearing in the table are defined in \S 4 to \S 10.

\begin{theorem}\label{T: Main Theorem}
    Let $D$ be a $\fpgmt$. Then $D\simeq D_{pcris}^{*}(V_2(E))$ for some elliptic curve $E/\Q_2$ with potentially good reduction if and only if $D$ is represented by one of the $\fpgmt$s listed in Table~\ref{Ta: Main} and defined in \S4-\S10.
\end{theorem}

\begin{table}[ht]
    \centering
    \begin{tabular}{|c|c|c|c|l|c|c|}
        \hline
        $e$ & Reduction type                             & $K$                     & Frobenius            & $\Gal(K/\Q_2)$ & $D_{cris,K}^{*}(V_2(E))$                                   & $\#\alpha$    \\
        \hline

        $1$ & Supersingular                              & $\Q_2$
            & $a_2=0$                                    & $1$                     & $D_{c}(1,0)$                    &                 $1$                                                       \\
        \cline{4-7}
            &                                            &
            & $a_2=\pm 2$                                & $1$                     & $D_{c}(1,\pm2)$                    &                 1                                                         \\
        \cline{2-7}
            & Ordinary                                   & $\Q_2$
            & $a_2=\pm1$                                 & $1$                     & $D_{c}(1,\pm1,\alpha)$                    &                 2                                                         \\
        \hline

                $2$ & Supersingular & $\Q_2(\sqrt{2})$
            & $a_2=0$ & $C_2$
            & $D_c(1,0)\otimes D_{\varepsilon_2}$ & 1 \\
        \cline{4-7}
            & &
            & $a_2=\pm2$ & $C_2$
            & $D_c(1,\pm2)\otimes D_{\varepsilon_2}$ & 1 \\
        \cline{3-7}
            & & $\Q_2(\sqrt{-2})$
            & $a_2=0$ & $C_2$
            & $D_c(1,0)\otimes D_{\varepsilon_{-2}}$ & 1 \\
        \cline{4-7}
            & &
            & $a_2=\pm2$ & $C_2$
            & $D_c(1,\pm2)\otimes D_{\varepsilon_{-2}}$ & 1 \\
        \cline{3-7}
            & & $\Q_2(\sqrt{-1})$
            & $a_2=0$ & $C_2$
            & $D_c(1,0)\otimes D_{\varepsilon_{-1}}$ & 1 \\
        \cline{4-7}
            & &
            & $a_2=\pm2$ & $C_2$
            & $D_c(1,\pm2)\otimes D_{\varepsilon_{-1}}$ & 1 \\
        \cline{2-7}
            & Ordinary & $\Q_2(\sqrt{2})$
            & $a_2=\pm1$ & $C_2$
            & $D_c(1,\pm1,\alpha)\otimes D_{\varepsilon_2}$ & 2 \\
        \cline{3-7}
            & & $\Q_2(\sqrt{-2})$
            & $a_2=\pm1$ & $C_2$
            & $D_c(1,\pm1,\alpha)\otimes D_{\varepsilon_{-2}}$ & 2 \\
        \cline{3-7}
            & & $\Q_2(\sqrt{-1})$
            & $a_2=\pm1$ & $C_2$
            & $D_c(1,\pm1,\alpha)\otimes D_{\varepsilon_{-1}}$ & 2 \\
        \hline

        $3$ & Supersingular                              & $\Q_2(\sqrt[3]{2},\zeta_3)$
            & $a_2=0$                                    & $D_3$                   & $D_{pc}(3,0,\alpha)$                    &                 $\PP^1(\Q_2)$                                             \\
        \hline

        $4$ & Supersingular                              & $K^{na}_{4,i}~1\leq i\leq 2$
            & $a_2=0$                                    & $D_4$                   & $D_{pc}(4,0,i,\alpha)$ &                 $\PP^1(\Q_2)$                                             \\
        \cline{3-7}
            &                                            &
        \multirow{2}{*}{$K^{ab}_{4,i}~1\leq i\leq 2$}
            & \multirow{2}{*}{$a_2=\pm2$}
            & \multirow{2}{*}{$C_4$}
            & $D_{pc}(4,\pm2,i,-)$                         &  1                                                                                                 \\
        \cline{6-7}
            &                                            &                         &                      &                & $D_{pc}(4,\pm2,i,+)$      &  1             \\
        \hline

        $6$ & Supersingular                              & $K_{6,i}~1\leq i\leq 3$
            & $a_2=0$                                    & $D_6$
            & $D_{pc}(6,0,i,\alpha)$                       &  $\PP^1(\Q_2)$                                                                                     \\
        \hline

        \multirow{2}{*}{$8$}
            & \multirow{2}{*}{Supersingular}
            & \multirow{2}{*}{$K_{8,i}~1\leq i\leq 6$}
            & \multirow{2}{*}{$a_2=0$}
            & \multirow{2}{*}{$\operatorname{SD}_{16}$}
            & $D_{pc}(8,0,i,\alpha,+)$                     &  $\PP^1(\Q_2)$                                                                                     \\
        \cline{6-7}
            &                                            &                         &                      &                & $D_{pc}(8,0,i,\alpha,-)$  &  $\PP^1(\Q_2)$ \\
        \hline

        \multirow{2}{*}{$24$}
            & \multirow{2}{*}{Supersingular}
            & \multirow{2}{*}{$K_{24,i}~1\leq i\leq 8$ }
            & \multirow{2}{*}{$a_2=0$}
            & \multirow{2}{*}{$\GL_2(\ma F_3)$}
            & $D_{pc}(24,0,i,\alpha,+)$                    &  $\PP^1(\Q_2)$                                                                                     \\
        \cline{6-7}
            &                                            &                         &                      &                & $D_{pc}(24,0,i,\alpha,-)$ &$\PP^1(\Q_2)$ \\
        \hline
    \end{tabular}
    \caption{Summary of the $\fpgmt$s of elliptic curves $E/\Q_2$ with potentially good reduction. The fields $K_{e,i}$ are the Galois closures of the corresponding fields $L$ in Table~\ref{Ta: InTypes}. The last column gives the parameter space of filtration classes when all the other parameters are fixed}
    \label{Ta: Main}
\end{table}

In \S2, we review the definitions and properties of Fontaine's functors and elliptic curves that we will need. In \S3, we combine results from these two theories to derive several properties of a semilinear nature that $D_{pcris}^{*}(V_2(E))$ must satisfy. In \S4--\S10, we determine all possible $\fpgmt$ satisfying these properties and study some relations among them, including which objects are twists of one another. Finally, in \S11, for each $\fpgmt$ $D$, we construct an elliptic curve $E$ such that $D_{pcris}^{*}(V_2(E))\simeq D$. To this end, we first introduce some basic notions concerning Dieudonné modules, together with Volkov's Galois pairs and her descent criterion for elliptic curves. These pairs and the descent criterion play a central role in constructing a curve for each module. In contrast with the case $p\geq 3$ the classification can be deduce without imposing the well know necessary condition of rational traces of the Weil-Deligne representation associated to $\fpgmt$s. This property is deduced a posteriori from the classification. This is due to the small number of possible $a_2$ for $E/\Q_2$.

\subsection{Acknowledgments} The authors would like to thank Maja Volkov, Filip Gawron and Josu Perez for helpful conversations that contributed to the development of this paper. Part of this work was carried out while both authors were visiting the Max Planck Institute for Mathematics in Bonn. The authors are grateful to the Institute for its hospitality and financial support.
\subsection{AI disclosure}
The mathematical work and the corresponding code were carried out without
using AI tools. Subsequently GPT-6 Astra and GPT Sol 6.1 were used for auditing mathematical arguments, the computations and to assist with revisions
and exposition. The authors are fully responsible for the final version.

\section{Background}\label{S:Background} 
\subsection{Notation}\label{Notation}

Let $K/\Q_p$ be a finite extension, $\cl O_{K}$ its ring of integers, and $K_0$ the maximal unramified
subextension of $K/\Q_p$. If $s=[K_0:\Q_p]$, then the residue field of $K$ is $\F_{q}$ with $q=p^s$. Let $\sigma$ be the Frobenius of $K_0/\Q_p$, that is, $\sigma(a)\equiv a^p\pmod p$ for all $a\in\cl O_{K_0}$, and
$\Gal(K_0/\Q_p)=\langle\sigma\rangle$. Let $\overline{K}$ be an algebraic closure, $G_{K}:=\Gal(\overline{K}/K)$ its Galois group, and $K^{un}$ the maximal unramified extension inside $\overline{K}$. We denote by $\chi_p:G_{\Q_p}\rightarrow \Q_p^{\times}$ the $p$-adic cyclotomic character whose defining property is $g(\zeta_{p^n})=\zeta_{p^n}^{\chi_p(g)}$. When regarded as a $\Q_p$-vector space, we denote it by $\Q_p(1)$.

Given a Galois extension $F/\Q_p$, an element $g\in \Gal(F/\Q_p)$, and a scheme $X/F$, we denote by $X^{g}$ the base change of $X$ along $g$, that is, $X^{g}:=X\times_{\Spec(F),g}\Spec(F)$, and we call it the twist of $X$ by $g$. Similarly, for a $p$-divisible group $\cl G=\lim \cl G_n$, the twist is obtained by twisting each of the $\cl G_n$ as well as the transition maps. For an $F$-vector space $V$, the twist is $V^{g}=V\otimes_{F,g}F$, whose defining relation is $\lambda v\otimes \mu=v\otimes g(\lambda)\mu$. The map $m_g:V\rightarrow V^g$ defined by $m_g(v)=v\otimes 1$ is $g$-semilinear, and, given a linear map $f:V^{g}\rightarrow W$, the map $f\circ m_g:V\rightarrow W$ is $g$-semilinear. We will use this identification between linear maps from $V^g$ and semilinear maps from $V$ throughout the paper.

Let $W_K$ be the Weil group of $K$ and $I_K$ its inertia subgroup. The group
$W_K$ is made into a topological group by declaring that $I_K$ is open in $W_K$
and that the subspace topology of $I_K$ in~$W_K$ is the same as the subspace
topology of $I_K$ in $\Gal(\overline{K}/K)$. With this topology, we have the short exact sequence of topological groups
$1\rightarrow I_K\rightarrow
    W_K \rightarrow \Z\rightarrow 1$, where the last map $d:W_K\rightarrow \Z$
is defined so that $g\in W_K$ acts on $\barra{\F}_{q}$ as $x\mapsto x^{{q}^{d(g)}}$ or, equivalently, $\bar{g}=Frob_q^{d(g)}\in \Gal(\overline{\F_{q}}/\F_q)$.

Let $F/\Q_p$ be a finite extension and let $E/F$ be an elliptic curve with potentially good reduction. Let $K/F$ be a Galois extension over which $E$ attains good reduction, minimal with this property. Then the semistability defect $e$ of $E$ is $\vert I_{K/F}\vert$. Equivalently, $e$ is the extension degree $[KF^{un}:F^{un}]$.

Finally, we denote by $Q_8$ the quaternion group of order 8 and by
$\SD_{16}$ the Sylow 2-subgroup of $\GL_2(\F_3)$.

\subsection{Fontaine functors}

This subsection is a brief introduction to the Fontaine functor $D_{pcris}^{*}$.
Readers familiar with it may skip this subsection. We first recall the covariant equivalence for crystalline $G_K$-representations, then add descent data for potentially crystalline $G_{\Q_p}$-representations, and finally dualize, since the contravariant functor is better suited to our purposes.

The functor $D_{pcris}^{*}$ provides an antiequivalence of categories between potentially crystalline representations of $G_{\Q_p}$ and a category of semilinear objects introduced at the end of this subsection. After recalling some definitions, we first introduce the functor $D_{cris,K}$, which provides an equivalence of categories between representations of $G_{\Q_p}$ that become crystalline over $K$ and weakly admissible $\fpgm$s.

\begin{defn}
    Let $K/\Q_p$ be a finite extension. A filtered $\varphi$-module over $K$ is a finite-dimensional
    $K_0$-vector space $D$ equipped with:
\begin{itemize}
        \item A $\sigma$-semilinear bijection $\varphi:D\longrightarrow D$, where $\sigma$ is the Frobenius of $K_0/\Q_p$.

        \item An exhaustive and separated filtration on $D_K:=D\otimes_{K_0}K$, that is, a
              sequence of subspaces $\Fil^{i}D_K\subset D_K$ indexed by $i\in \Z$ such that
              $\Fil^{i+1}D_K\subset \Fil^{i}D_K$ and such that $\Fil^iD_K=D_K$ for $i<<0$ and
              $\Fil^iD_K=0$ for $i>>0$.
    \end{itemize}
    A morphism of filtered $\varphi$-modules from $D$ to $D'$ is a $K_0$-linear map $f:D\longrightarrow D'$ such that $f\circ \varphi=\varphi'\circ f$ and $f\otimes \Id(\Fil^i D_K)\subset \Fil^i D'_K$.
\end{defn}

The category of filtered $\varphi$-modules is denoted by $MF_K^{\varphi}$. To each filtered $\varphi$-module $D$, we associate the following
quantities:

\begin{itemize}
    \item $t_N(D)=v_p(\lambda)$, where $\lambda$ is the element defined by $\Lambda^n\varphi(x)=\lambda x$ for any nonzero $x\in\Lambda^n D$, with $n:=\dim_{K_0}D$. It can be shown that $t_N(D)$ is independent of the choice of $x$.

    \item $t_H(D)=\sum_{i}i\cdot \dim_K(\Fil^i D_K/\Fil^{i+1}D_K)$.
\end{itemize}

We say that $D\in MF^{\varphi}_K$ is weakly admissible if $t_N(D)=t_H(D)$ and
every subobject $\tilde{D}\subset D$ in $MF^{\varphi}_K$ satisfies
$t_N(\tilde{D})\geq t_H(\tilde{D})$. The weakly admissible objects $D\in
    MF^{\varphi}_K$ form an abelian category denoted by $MF^{\varphi,w.a}_K$. The functor
\begin{align*}
    D_{cris,K}:Rep^{cris}(G_K) & \longrightarrow MF^{\varphi,w.a}_K             \\
    V                          & \longrightarrow D_{cris,K}(V):=(B_{cris}\otimes_{\Q_p}V)^{G_K}
\end{align*}
is an equivalence of categories, where $Rep^{cris}G_K$ is the category of crystalline representations of $G_K$ and $B_{cris}$ is the crystalline period ring; see~\cite{Colmez-Fontaine} for details of its construction and properties.

The object $K\{n\}\in MF^{\varphi,w.a.}_{K}$ is the vector space $D=K_0e_1$ with $\varphi(e_1)=p^{-n}e_1$ and filtration $\Fil^{i}D_K=Ke_1$ for $i\leq -n$ and $\Fil^{i}D_K=0$ for $i\geq -n+1$. This object is important because, for $\chi_p$, the $p$-adic cyclotomic character, we have $D_{cris,\Q_p}(\chi_p^n)=\Q_p\{n\}$.

Let $V$ be a $G_{\Q_p}$-representation that becomes crystalline over the Galois extension $K/\Q_p$. Since $G_{\Q_p}$ acts on $B_{cris}\otimes V$, the $K_0$-vector space $D_{cris,K}(V)=(B_{cris}\otimes V)^{G_K}$ carries a natural action of $\Gal(K/\Q_p)$. Considering $D_{cris,K}(V)$ together with this $\Gal(K/\Q_p)$-action allows us to recover the potentially crystalline representation $V$. This motivates the definition of the following category.

\begin{defn} Let $K/\Q_p$ be a finite Galois extension. Then
    a $\fpgm$ is a finite-dimensional $K_0$-vector space $D$ together with:

    \begin{itemize}
        \item A $\sigma$-semilinear bijective Frobenius $\varphi:D \longrightarrow D$.
        \item A semilinear action of $\Gal(K/\Q_p)$, i.e., $g(\lambda
                  v)=g\vert_{K_0}(\lambda)\cdot g(v)$.
              Moreover, this action must commute with the Frobenius action, $\varphi g=g
                  \varphi$.

        \item An exhaustive and separated filtration on $D_K:=D\otimes_{K_0} K$ by $\Gal(K/\Q_p)$-stable subspaces under the diagonal action, i.e., $g(v\otimes
                  \lambda)=g(v)\otimes g(\lambda)$ for $v\in D$ and $\lambda\in K$.

        \item If we forget the Galois action, $D$ is weakly admissible as an object in $MF_{K}^{\varphi}$.
    \end{itemize}
    A morphism of $\fpgm$s from $D$ to $D'$ is a morphism of filtered $\varphi$-modules that also commutes with the action of $\Gal(K/\Q_p)$, i.e., $f\circ g_D=g_{D'} \circ f$ for all $g\in \Gal(K/\Q_p)$.

\end{defn}

\begin{remark}
    The last condition in the definition is nonstandard, but including it simplifies the exposition, since objects $D$ that do not satisfy it play no role in the paper.
 \end{remark}

The functor $D_{cris,K}$ induces an equivalence of categories
between representations of $G_{\Q_p}$ that become crystalline
over $K$ and weakly admissible $\fpgmt$s. Moreover, it is compatible
with tensor products and duals:
 $$D_{cris,K}(V_1\otimes_{\Q_p}V_2)
    \simeq D_{cris,K}(V_1)\otimes_{K_0}D_{cris,K}(V_2),
    \qquad
    D_{cris,K}(V_1^\vee)\simeq D_{cris,K}(V_1)^\vee.$$
For two $\fpgmt$s $D_1$ and $D_2$, these operations are defined
as follows:
\begin{itemize}
    \item The tensor product is $D_1\otimes_{K_0}D_2$, with
    Frobenius and Galois action
    $$\varphi_1\otimes\varphi_2\qquad g(d_1\otimes d_2)=g(d_1)\otimes g(d_2),$$
    and filtration
    $$\Fil^i(D_1\otimes_{K_0}D_2)_K=\sum_{j\in\Z}\Fil^jD_{1,K}\otimes_K\Fil^{i-j}D_{2,K}.$$

    \item The dual of $D_1$ is
    $D_1^\vee:=\Hom_{K_0}(D_1,K_0)$, with Frobenius
    and Galois action given by
    $$\varphi_{D_1^\vee}(f)=\sigma\circ f\circ\varphi_1^{-1},\qquad(g\cdot f)(d)=g\bigl(f(g^{-1}d)\bigr).$$
    Under the natural identification
    $(D_1^\vee)_K\simeq\Hom_K(D_{1,K},K)$, its filtration is
        $$\Fil^i(D_1^\vee)_K=\bigl\{f\in\Hom_K(D_{1,K},K):f(\Fil^{1-i}D_{1,K})=0\bigr\}.$$
\end{itemize}

In practice, we also need to work with representations that become crystalline over different extensions. For example, when
twisting a representation $\rho$ by a character $\varepsilon$, one wants to
work with representations $\rho_1,\rho_2$ of $G_{\Q_p}$ that become crystalline
over different extensions $K_1,K_2$. Since both representations become
crystalline over the compositum $K_1K_2$, we can work with them as
$\text{filtered}~(\varphi,\Gal(K_1K_2/\Q_p))-\text{modules}$, where it makes
sense to take the tensor product. Moreover, it may happen that $\rho_1\otimes\rho_2$ becomes crystalline over a subextension of the compositum $K_1K_2$. Thus, to work with potentially crystalline representations in general, we need to relate $D_{cris,K}(V)$ to $D_{cris,K'}(V)$ for different Galois extensions $K/\Q_p$ and $K'/\Q_p$.

Let $K'/\Q_p$ be a Galois extension and $K/\Q_p$ a Galois subextension. Let $K'_0$ be the maximal unramified subextension
inside $K'$ and $\sigma'$ the Frobenius of the extension $K_0'/\Q_p$. If $V$ is a
 $G_{\Q_p}$-representation that becomes crystalline over $K$, it also becomes crystalline over $K'$. We have an
inclusion $D_{cris,K}(V)\hookrightarrow D_{cris,K'}(V)$ and, by the
definition of the functor, we have
$D_{cris,K'}(V)^{\Gal(K'/K)}=D_{cris,K}(V)$. Conversely, if $D_{cris,K}(V)=D$, then $D_{cris,K'}(V\vert_{G_{K'}})$ is the vector space $D':=D\otimes_{K_0}K'_0$ with
Frobenius $\varphi'=\varphi\otimes \sigma'$ and filtration
$\Fil^{i}D'_{K'}=\Fil^{i}D_K\otimes_K K'$. Finally, if we know how $\Gal(K/\Q_p)$ acts on
$D_{cris,K}(V)$, then the action of $\Gal(K'/\Q_p)$ on $D'$ is given by $g\cdot
    (v\otimes\lambda)=\bar{g}(v)\otimes g(\lambda)$, where
$\bar{g}$ is the image of $g$ in $\Gal(K/\Q_p)$.

This leads to the functor $D_{pcris}$, which provides a
unified way to describe potentially crystalline representations independently of
the extension $K$ over which they become crystalline. By Fontaine's general theory, $D_{cris,K}(V)$ is a $B_{cris}^{G_K}=K_0$-vector space. It follows
that $$ D_{pcris}(V)=\dirlim{[K:\Q_p]<\infty} (B_{cris}\otimes
    V)^{G_{K}}=\dirlim{[K:\Q_p]<\infty} D_{cris,K}(V) $$ is a
$\dirlim{[K:\Q_p]<\infty}B_{cris}^{G_{K}}=\Q^{un}_p$-vector space and
$\dim_{\Q^{un}_p}D_{pcris}(V) \leq \dim_{\Q_p}(V)$ with equality if and only
if $V$ is potentially crystalline. The functor $D_{pcris}$ is an equivalence of
categories between potentially crystalline Galois representations and weakly
admissible filtered $(\varphi,Gal(\overline{\Q}_p/\Q_p))$-modules for which the action of
$Gal(\bar{\Q}_p/\Q_p)$ is smooth, i.e., for every vector $v$ in the
$\Q^{un}_p$-vector space $D$, the subgroup $\operatorname{Stab}(v)$ is open. If $V$ is a
representation of $G_{\Q_p}$ that becomes crystalline over $K$, then
$D_{pcris}(V)$ is completely determined by the $\fpgm$ $D_{cris,K}(V)$ via the base changes described above. Thus, to describe a filtered $(\varphi,Gal(\overline{\Q}_p/\Q_p))$-module, it suffices to give the corresponding $\fpgm$ for a suitable $K$. We use this convention throughout the paper: we say that $D_{pcris}(V_2(E))$ is equal to a certain $\fpgm$ over an extension $K$ over which $E$ attains good reduction. This allows us to work with potentially crystalline objects in general.

\subsection{Elliptic curves}
Let $E/\Q_p$ be an elliptic curve, and set $T_{\ell}=\varprojlim E[\ell^n]$ and $V_\ell(E)=T_\ell(E)\otimes\Q_\ell$. The action of $G_{\Q_p}$ on $V_{\ell}$ defines a representation
$\rho_{E,\ell^{\infty}}:G_{\Q_p}\rightarrow\GL_2(\Q_{\ell})$. The behavior of this representation depends strongly on whether $\ell\neq p$ or $\ell=p$. When
$\ell\neq p$, one can construct from $\rho_{E,\ell^{\infty}}$ a Weil--Deligne
representation $(\rho_{WD},N):W_{\Q_p}\rightarrow \GL_2(\ma C)$; see~\cite{Rohrlich1994} for a detailed exposition. This
Weil--Deligne representation contains the same information as
$\rho_{E,\ell^{\infty}}$ and is independent of the prime $\ell\neq p$ chosen.

 Moreover, the Néron--Ogg--Shafarevich criterion states that $E$ attains good
reduction over the extension $K/\Q_p$ if and only if
$\rho_{E,\ell^{\infty}}\vert_{G_K}$ is unramified or, equivalently, if and only
if $N=0$ and $\rho_{WD}\vert_{W_K}$ is unramified. When $\ell=p$, the
representation $\rho_{E,p^{\infty}}$ is much more complex and less well understood.
For instance, the analogue of the Néron--Ogg--Shafarevich criterion states that $E$ has good
reduction over $K$ if and only if $\rho_{E,p^{\infty}}\vert_{G_K}$ is
crystalline, a notion much more involved than that of being unramified. Moreover, in order to
define the Weil--Deligne representation associated to $\rho_{E,p^{\infty}}$, one
needs to pass first through the potentially semistable Fontaine functor $D_{pst}$
and from there to $\rho_{WD}$. We give the construction when $E$
has potentially good reduction or, equivalently, when $\rho_{E,p^{\infty}}$ is
potentially crystalline. Let $D:=D_{cris,K}(V_p(E))$ be a $\fpgm$. Then the
associated Weil--Deligne representation is \begin{align*}
    \rho_{WD}(D):W_{\Q_p} & \longrightarrow Aut_{K_0}(D)\simeq \GL_2(K_0)\hookrightarrow\GL_2(\barra{\Q}_p)\simeq \GL_2(\ma C) \\
    g                     & \longrightarrow \barra{g}\cdot \varphi^{-d(g)}.
\end{align*}
where $\bar{g}$ is the class of $g$ in $\Gal(K/\Q_p)=W_{\Q_p}/W_K$.

 Finally, the Weil--Deligne representation associated to $\rho_{E,p^{\infty}}$ is $\rho_{WD}(D_{cris,K}(\rho_{E,p^{\infty}}))$. Note that the filtration of $D$ does not play any role in the definition of $\rho_{WD}(D)$; thus, unlike in the $\ell\neq p$ case, the Weil--Deligne representation retains less information than $\rho_{E,p^{\infty}}$. With this definition, the Weil--Deligne representation associated to $\rho_{E,\ell^{\infty}}$ is independent of $\ell$ for all $\ell$, including $\ell=p$. Two known properties of $\rho_{E,p^{\infty}}$ are that $\det \rho_{E,p^{\infty}}=\chi_p$, the cyclotomic character, and that its Hodge--Tate weights are $(0,1)$, that is, $\dim_{K} \Fil^{-1}(D_{cris,K}(V_p(E)))=2$, $\dim_{K} \Fil^{0}(D_{cris,K}(V_p(E)))=1$ and $\dim_{K} \Fil^1(D_{cris,K}(V_p(E)))=0$.

\subsection{Contravariant convention}
In practice, instead of using the functors $D_{cris,K}$ and $D_{pcris}$, we
use their duals $D_{cris,K}^{*}$ and $D_{pcris}^{*}$, defined by

$$ D_{cris,K}^{*}(V)\simeq (D_{cris,K}(V))^{\vee}\simeq D_{cris,K}(V^{\vee}) \simeq Hom_{\Q_p[G_{K}]}(V,B_{cris}).
$$

$$ D_{pcris}^{*}(V)=\dirlim{[K:\Q_p]<\infty}D_{cris,K}^{*}(V).$$

For a potentially crystalline $G_{\Q_p}$-representation $V$ that becomes crystalline over $K$, we have 

$$D_{pcris}^{*}(V)=\Q_p^{un}\otimes D_{cris,K}^{*}(V)\simeq (\Q_p^{un}\otimes D_{cris,K}(V))^{\vee}\simeq D_{pcris}(V)^{\vee}$$

This choice of convention is motivated by the fact that most results in $p$-adic Hodge theory are stated in
terms of $D_{cris,K}(H_{et}^1(E,\Q_p))$, which is isomorphic to
$D_{cris,K}(V_p(E)^{\vee})\simeq D_{cris,K}^*(V_p(E))$. For instance, the
crystalline--étale comparison isomorphism states that
$D_{cris,K}(H_{et}^{1}(E,\Q_p))\simeq H_{cris,K}^1(\tilde{E}/K_0)$, where $\tilde{E}$ is the special fiber of $E$. Katz and Messing showed in~\cite[Theorem 1]{KatzMessing1974} that, for an elliptic curve $E/\F_{p^s}$, the characteristic polynomial of the $p^s$-th power Frobenius acting on $H_{et}^1$ is the same as the characteristic
polynomial of the $s$-th power of the crystalline Frobenius, $\varphi^s$, acting on $H_{cris}$. Thus,
if the elliptic curve $E$ attains good reduction over a totally ramified
extension $L$, since $$D_{cris,L}^*(V_p(E))\simeq
    D_{cris,L}(H_{et}^1(E,\Q_p))\simeq H^1_{cris,L}(\tilde{E}/\Q_p)$$, we know that the
characteristic polynomial of $\varphi$ acting on $D_{cris,L}^*(V_p(E))\simeq
    D_{cris,K}^*(V_p(E))^{\Gal(K/L)}$ is $X^2-a_pX+p$. The use of $D_{pcris}^*$
changes the filtration so that now Hodge--Tate weights $(0,1)$ correspond to
$\dim_{K} \Fil^{0}(D_{cris,K}^{*}(V_p(E)))=2$, $\dim_{K}
    \Fil^{1}(D_{cris,K}^{*}(V_p(E)))=1$ and $\dim_{K} \Fil^2(D_{cris,K}^{*}(V_p(E)))=0$.

On the other hand, using $D_{pcris}^*$ instead of $D_{pcris}$ affects
the compatibility of the Weil--Deligne representations. It is not true that
$\rho_{WD}(D_{cris,K}^*(V_p(E)))$ is isomorphic to the Weil--Deligne
representations for $\ell\neq p$. However, since $V_p(E)^{\vee}\simeq V_p(E)\otimes\Q_p(-1)$, we have
$\rho_{WD}(D_{cris,K}^*(V_p(E)))\otimes\rho_{WD}(\Q_p\{1\}) \simeq WD_{\ell}$.
We will only use the compatibility of the Weil--Deligne representations for
different $\ell$ in order to deduce that $\rho_{WD}(D_{pcris}^{*}(V_p(E)))$ must
have traces in $\Q$ and to deduce properties of the action of inertia which are
not affected by the cyclotomic twist. Finally, recall
that subobjects in $MF^{\varphi,w.a.}_{K}$ transform under $D_{cris,K}^{*}$
into quotients and vice versa.

\section{Setup for $\Q_2$}

In this section, we establish the preliminary lemmas needed for $\Q_2$. We first need a list of Galois extensions $K/\Q_2$ such that any elliptic curve $E/\Q_2$ with potentially good reduction attains it over a field $K$ in the list. Such a list is provided by the inertia types of elliptic curves over $\Q_2$ described in~\cite{dembélé2024galoisinertialtypeselliptic}. We then define a class of $\upgmt$s with properties modeled on those arising from elliptic curves and prove several lemmas about them.

Let $K/\Q_2$ be a Galois extension over which $E$ attains good reduction, minimal with this property. Such an extension need not be unique. Suppose that $D_{pcris}^{*}(V_2(E))$ is represented as a $\fpgmt$. If $F$ is another Galois extension over which $E$ attains good reduction, the construction in \S\ref{S:Background} gives the following description of $D_{pcris}^{*}(V_2(E))$ as a filtered $(\varphi,\Gal(F/\Q_2))$-module:
$D^{*}_{cris,F}(V_2(E))=[D^{*}_{cris,K}(V_2(E))\otimes_{K_0}(FK)_0]^{\Gal(FK/F)}$. Both $D^{*}_{cris,K}(V_2(E))$ and $D^{*}_{cris,F}(V_2(E))$ represent the same filtered $(\varphi,\Gal(\barra{\Q}_2/\Q_2))$-module and therefore correspond to the same potentially crystalline representation. Thus, it suffices to compute a $\fpgmt$ for $K$ equal to the Galois closure of each of the fields $L$ in Table~\ref{Ta: InTypes}.

\begin{lemma}\label{L: InTypes}
    Let $E/\Q_2$ be an elliptic curve with potentially good reduction and inertia type $\tau_E$. Then $\tau_E$ is isomorphic to one of the inertia types $\tau$ in Table~\ref{Ta: InTypes}, with semistability defect $e$ and conductor exponent $\nu_2(N_E)$. Moreover, $E$ attains good reduction over the corresponding extension $L/\Q_2$ whose Galois closure $K/\Q_2$ has inertia $I:=I_{K/\Q_2}$ and Galois group $G:=\Gal(K/\Q_2)$ as described in Table~\ref{Ta: InTypes}.

  \begin{table}[htbp]
    \centering
\begin{tabular}{c|c|c|c|c|c|c}
    $\tau$ & $e$ & $v_2(N_E)$ & $L$
    & Defining polynomial & $I$ & $G$ \\
    \hline\hline

    trivial & $1$ & $0$ & $\Q_2$
    & \href{https://www.lmfdb.org/padicField/2.1.0.1}{$x$}
    & $1$ & $1$ \\
    \hline

    $\varepsilon_{-1}$ & \multirow{3}{*}{$2$} & $4$ & $\Q_2(\sqrt{-1})$
    & \href{https://www.lmfdb.org/padicField/2.2.2.1}{$x^2+1$}
    & \multirow{3}{*}{$C_2$} & \multirow{3}{*}{$C_2$} \\
    \cline{1-1}\cline{3-5}

    $\varepsilon_2$ &  & \multirow{2}{*}{$6$} & $\Q_2(\sqrt{2})$
    & \href{https://www.lmfdb.org/padicField/2.2.3.1}{$x^2-2$}
    &  &  \\
    \cline{1-1}\cline{4-5}

    $\varepsilon_{-2}$ &  &  & $\Q_2(\sqrt{-2})$
    & \href{https://www.lmfdb.org/padicField/2.2.3.3}{$x^2+2$}
    &  &  \\
    \hline

    $\tau_{sc,2}(5,1,3)$ & $3$ & $2$ & $\Q_2(\sqrt[3]{2})$
    & \href{https://www.lmfdb.org/padicField/2.3.2.1}{$x^3+2$}
    & $C_3$ & $D_3$ \\
    \hline

    $\tau_{sc,2}(5,4,4)$ & \multirow{4}{*}{$4$} & \multirow{4}{*}{$8$} & $L^{na}_{4,1}$
    & \href{https://www.lmfdb.org/padicField/2.4.11.18}{$x^4+8x+18$}
    & \multirow{4}{*}{$C_4$} & \multirow{2}{*}{$D_4$} \\
    \cline{1-1}\cline{4-5}

    $\tau_{sc,2}(5,4,4)\otimes\varepsilon_{-4}$ &  &  & $L^{na}_{4,2}$
    & \href{https://www.lmfdb.org/padicField/2.4.11.17}{$x^4+4x^2+6$}
    &  &  \\
    \cline{1-1}\cline{4-5}\cline{7-7}

    $\tau_{ps,2}(1,4,4)$ &  &  & $L^{ab}_{4,1}$
    & \href{https://www.lmfdb.org/padicField/2.4.11.1}{$x^4-4x^2+2$}
    &  & \multirow{2}{*}{$C_4$} \\
    \cline{1-1}\cline{4-5}

    $\tau_{ps,2}(1,4,4)\otimes\varepsilon_{-4}$ &  &  & $L^{ab}_{4,2}$
    & \href{https://www.lmfdb.org/padicField/2.4.11.2}{$x^4+4x^2+2$}
    &  &  \\
    \hline

    $\tau_{sc,2}(5,1,3)\otimes\varepsilon_{-4}$ & \multirow{3}{*}{$6$} & $4$ & $L_{6,1}$
    & \href{https://www.lmfdb.org/padicField/2.6.8.1}{$x^6+2x^3+2$}
    & \multirow{3}{*}{$C_6$} & \multirow{3}{*}{$D_6$} \\
    \cline{1-1}\cline{3-5}

    $\tau_{sc,2}(5,1,3)\otimes\varepsilon_8$ &  & \multirow{2}{*}{$6$} & $L_{6,2}$
    & \href{https://www.lmfdb.org/padicField/2.6.11.1}{$x^6+14$}
    &  &  \\
    \cline{1-1}\cline{4-5}

    $\tau_{sc,2}(5,1,3)\otimes\varepsilon_{-8}$ &  &  & $L_{6,3}$
    & \href{https://www.lmfdb.org/padicField/2.6.11.9}{$x^6+2$}
    &  &  \\
    \hline

    $\tau_{sc,2}(-20,3,4)$ & \multirow{6}{*}{$8$} & \multirow{2}{*}{$5$} & $L_{8,1}$
    & \href{https://www.lmfdb.org/padicField/2.8.16.65}{$x^8+2x^6+4x^4+4x^3+4x+2$}
    & \multirow{6}{*}{$Q_8$} & \multirow{6}{*}{$\SD_{16}$} \\
    \cline{1-1}\cline{4-5}

    $\tau_{sc,2}(-4,3,4)$ &  &  & $L_{8,2}$
    & \href{https://www.lmfdb.org/padicField/2.8.16.66}{$x^8+6x^4-3$}
    &  &  \\
    \cline{1-1}\cline{3-5}

    $\tau_{sc,2}(-20,3,4)\otimes\varepsilon_8$ &  & \multirow{2}{*}{$6$} & $L_{8,3}$
    & \href{https://www.lmfdb.org/padicField/2.8.18.74}{$x^8-84x^4+392x^2-588$}
    &  &  \\
    \cline{1-1}\cline{4-5}

    $\tau_{sc,2}(-4,3,4)\otimes\varepsilon_8$ &  &  & $L_{8,4}$
    & \href{https://www.lmfdb.org/padicField/2.8.18.73}{$x^8-6x^4-3$}
    &  &  \\
    \cline{1-1}\cline{3-5}

    $\tau_{sc,2}(-4,6,4)\otimes\varepsilon_8$ &  & \multirow{2}{*}{$8$} & $L_{8,5}$
    & \href{https://www.lmfdb.org/padicField/2.8.24.66}{$x^8-12x^4-12$}
    &  &  \\
    \cline{1-1}\cline{4-5}

    $\tau_{sc,2}(-4,6,4)$ &  &  & $L_{8,6}$
    & \href{https://www.lmfdb.org/padicField/2.8.24.68}{$x^8+12x^4-12$}
    &  &  \\
    \hline

    $\tau_{ex,2,1}$ & \multirow{8}{*}{$24$} & $3$ & $L_{24,1}$
    & $x^{24}+2x^9+2x^6+2$
    & \multirow{8}{*}{\footnotesize $\SL_2(\F_3)$} & \multirow{8}{*}{\footnotesize $\GL_2(\F_3)$} \\
    \cline{1-1}\cline{3-5}

    $\tau_{ex,2,1}\otimes\varepsilon_{-4}$ &  & $4$ & $L_{24,2}$
    & $x^{24}+2x^{15}+2x^6+2$
    &  &  \\
    \cline{1-1}\cline{3-5}

    $\tau_{ex,2,1}\otimes\varepsilon_8$ &  & \multirow{2}{*}{$6$} & $L_{24,3}$
    & $x^{24}+4x^{15}+2x^6+4x^3+2$
    &  &  \\
    \cline{1-1}\cline{4-5}

    $\tau_{ex,2,1}\otimes\varepsilon_{-8}$ &  &  & $L_{24,4}$
    & $x^{24}+4x^{21}+2x^6+4x^3+2$
    &  &  \\
    \cline{1-1}\cline{3-5}

    $\tau_{ex,2,2}$ &  & \multirow{4}{*}{$7$} & $L_{24,5}$
    & $x^{24}+4x^{21}+4x^6+8x^3+2$
    &  &  \\
    \cline{1-1}\cline{4-5}

    $\tau_{ex,2,2}\otimes\varepsilon_{-4}$ &  &  & $L_{24,6}$
    & $x^{24}+4x^{21}+8x^9+4x^6+8x^3+2$
    &  &  \\
    \cline{1-1}\cline{4-5}

    $\tau_{ex,2,2}\otimes\varepsilon_8$ &  &  & $L_{24,7}$
    &  $x^{24}+4x^{21}+4x^6+2$
    &  &  \\
    \cline{1-1}\cline{4-5}

    $\tau_{ex,2,2}\otimes\varepsilon_{-8}$ &  &  & $L_{24,8}$
    &  $x^{24}+4x^{21}+8x^9+4x^6+2$
    &  &  \\
\end{tabular}
    \caption{Inertia types of elliptic curves over $\Q_2$.}
    \label{Ta: InTypes}
\end{table}\end{lemma}
\begin{proof}
    All assertions except those concerning the polynomials for the exceptional types follow from the main theorem of \cite{dembélé2024galoisinertialtypeselliptic}. These polynomials were computed using the Magma package in~\cite{CastroMorenoFloritFreitas2026}. Their correctness can be checked by verifying that the curves $E$ in~\cite{dembélé2024galoisinertialtypeselliptic} with exceptional types attain good reduction over the field defined by the corresponding polynomial; see~\cite{MAGMAFiles}.
\end{proof}

\begin{defn}
    Given $E/\Q_2$ with potentially good reduction, let $L$ in Table~\ref{Ta: InTypes} be field corresponding to $\tau_E$, the inertia type of $E$, then $E$ has good reduction over $L$ and $L$ is minimal with this property. In this case, we say that $E$ has the inertia type associated to $L$. 
\end{defn}

\begin{lemma}\label{L: t exists}
    Let $L$ be a totally ramified extension and let $K$ be its Galois closure, sssume moreover that $K=L\cdot K_0$. If $[K_0:\Q_2]=f$, then there exists an element $t\in\Gal(K/\Q_2)$ such that $\Gal(K/L)=\langle t\rangle$ and $t\vert_{K_0}=\sigma$. Moreover, $\Gal(K/\Q_2)=I_{K/\Q_2}\rtimes\langle t\rangle$.
\end{lemma}
\begin{proof}
 Since $K=L\cdot K_0$ and $K_0\cap L=\Q_2$, the restriction map $\Gal(K/L)\rightarrow \Gal(K_0/\Q_2)$ is an isomorphism. Let $t$ be the preimage of $\sigma$. Then $t\vert_{K_0}=\sigma$ and $\langle t \rangle=\Gal(K/L)$. Moreover, $\Gal(K/\Q_2)=I_{K/\Q_2}\cdot \langle t\rangle$ and $I_{K/\Q_2}\cap\langle t\rangle=1$. Since $I_{K/\Q_2}$ is normal, this implies $\Gal(K/\Q_2)=I_{K/\Q_2}\rtimes\langle t\rangle$.
\end{proof}

\begin{remark}
    All the fields $L$ in Table~\ref{Ta: InTypes} satisfy the condition of Lemma~\ref{L: t exists} with either $f=1$ or $f=2$.
\end{remark}

Let $D$ be an $\upgmt$, let $f$ be a $g$-semilinear endomorphism and let $e_1,e_2$ be a basis. For $v=c_1e_1+c_2e_2$, we have $f(v)=g(c_1)f(e_1)+g(c_2)f(e_2)$. Thus, the matrix $M(f)$ whose columns are the coordinates of $f(e_1),f(e_2)$ determines the image of any vector. Indeed, the coordinates of $f(v)$ are given by $M(f)g(v)$. After fixing a basis, the set of $\Gal(K_0/\Q_2)$-semilinear maps of $D$ is therefore isomorphic to $M_2(K_0)\rtimes\Gal(K_0/\Q_2)$, where the composition of the $g_i$-semilinear maps $f_i$ is given by 
$$(M(f_1),g_1)\cdot (M(f_2),g_2)=(M(f_1)g_1(M(f_2)),g_1\cdot g_2)$$ 

For instance, once we fix a basis, the condition that the $\sigma$-semilinear map $\varphi$ commutes with the action of $g\in \Gal(K/\Q_2)$ reads $$M(\varphi)\sigma(M(g))=M(g)\overline{g}(M(\varphi))$$

Here $\overline{g}$ is the class of $g\in \Gal(K_0/\Q_2)$. It is important to distinguish semilinear maps from their matrices. For example, below we encounter a $\sigma$-semilinear map $t$ whose matrix is $M(t)=\Id$. This does not mean that $t=\Id$, only that $t(e_i)=e_i$ for the basis vectors; for instance, $t(\zeta_3e_1)=\zeta_3^2 e_1$. We will use this distinction throughout the arguments involving both maps and matrices.

The following definition captures the properties of an $\upgmt$ arising from an elliptic curve. We prove our classification result for modules satisfying this condition so that we can apply it to the $\upgmt$s arising from Galois pairs and show that they agree with those arising from elliptic curves.

\begin{defn}
    Let $L/\Q_2$ be a totally ramified extension and $K$ its Galois closure. Assume that $K=L\cdot K_0$, and let $t$ be as in Lemma~\ref{L: t exists}. Let $D$ be an $\upgmt$ and $D_0=D^{\langle t \rangle}$. We say that $D$ satisfies condition ${\bf(A)}$ for $K/L$ if\begin{itemize}
    \item $\dim_{K_0}D=2$ 
    \item $I_{K/\Q_2}$ acts faithfully.
    \item For all $g\in I_{K/\Q_2}$, the linear operator $g$ satisfies $\det g=1$. 
    \item The operator $\varphi_{0}:=\varphi\vert_{D_0}$ has characteristic polynomial $X^2-a_2X+2$ with $a_2^2-8\neq 0$
    \end{itemize} 
    If, moreover, $[K_0:\Q_2]={\bf r}$, then we say that $D$ satisfies condition ${\bf (A.r)}$ for $K/L$.
\end{defn}

We now give several lemmas about the fields in Table~\ref{Ta: InTypes} and the corresponding $\upgmt$s satisfying condition ${\bf(A)}$.

\begin{lemma}\label{L: eigenvalues}
    If $D$ is an $\upgmt$ that satisfies ${\bf (A)}$ and $g\in I_{K/\Q_2}$ has order $m$, then $\operatorname{Tr}(g)=\zeta_m+\zeta_m^{-1}$, where $\zeta_m$ is some primitive root of unity. 
\end{lemma}
\begin{proof}
    Since $g^m=1$, the minimal polynomial of $g$ divides $X^m-1$. It is therefore separable, so $g$ is diagonalizable over some extension of $K_0$. Since $\det g=1$, the eigenvalues are $\lambda,\lambda^{-1}$, where $\lambda$ is an $m$-th root of unity. If $\lambda$ were not a primitive root of unity, the action of $g^n$ would be trivial for some $n<m$, contradicting the faithfulness of the action of $I_{K/\Q_2}$.  
\end{proof}
\begin{lemma}\label{L: adapted basis}
    If $D$ is an $\upgmt$ that satisfies ${\bf (A)}$ for $K/L$, then there exists a $\Q_2$-basis of $D_0$ that is a $K_0$-basis of $D$ in which $M(\varphi)=\tbt{0}{-2}{1}{a_2}$ and $M(t)=\tbt{1}{0}{0}{1}$.
\end{lemma}
\begin{proof}
    By the condition on the characteristic polynomial, $\varphi_0$ is not scalar, thus, there exist a vector $e_1$ of $D_0$ for which $(e_1,\varphi_0(e_1))$ is a basis. In this basis the expression of $M(\varphi)$ is the desired one. The expression for $M(t)$ follows from the fact that $e_1,\varphi(e_1)\in D^{\langle t\rangle}$. Since $D=D_0\otimes K_0$ it is clear that these element form a basis of $D$.  
\end{proof}
\begin{lemma}\label{L: a_2=0}
    If $D$ is a $\upgmt$ that satisfies ${\bf (A.2)}$ for $K/L$, then $a_2=0$. 
\end{lemma}
\begin{proof}
    We work in the basis provided by Lemma~\ref{L: adapted basis}.
    Since $[K_0:\Q_2]=2$ then $\varphi^2$ is linear and its describe by the matrix $M(\varphi^2)=M(\varphi)\sigma(M(\varphi))=M(\varphi)^2$. Let $g\in I_{K/\Q_2}$, since $g$ commutes with $\varphi$ it also commutes with $\varphi^2$ hence $M(\varphi)^2M(g)=M(g)M(\varphi)^2$. Moreover, the matrix equation for $g\varphi=\varphi g$ reads $M(\varphi)\sigma(M(g))=M(g)M(\varphi)$. Since $\varphi_0^2-a_2\varphi_0+2=0$ and the matrices of $\varphi_0$ and $\varphi$ are the same, if $a_2\neq 0$ we obtain $M(\varphi)=\frac{M(\varphi)^2+2}{a_2}$. Thus $M(g)$ commuting with $M(\varphi)^2$ implies that it also commutes with $M(\varphi)$, and hence $M(g)\in M_2(\Q_2)$. Then $M(g)M(t)=M(t)M(g)$, which implies that $gtg^{-1}t^{-1}\in I_{K/\Q_2}$ acts trivially on $D$. By the faithfulness of the action, $gtg^{-1}t^{-1}=1$. Thus $\Gal(K/L)$ is normal, so $L$ is Galois, a contradiction.
\end{proof}
\begin{lemma}\label{L: curves satisfy A}
    Let $L,K$ and $t$ be as described in Lemma~\ref{L: t exists} (in particular, $L$ can be any of the fields in Table~\ref{Ta: InTypes}), and let $E/\Q_2$ be an elliptic curve with inertia type associated to $L$ and whose reduction has trace of Frobenius $a_2$. Finally, let $D=D_{cris,K}^{*}(V_2(E))$ be an $\upgmt$. Then $D$ satisfies condition ${\bf (A)}$.
\end{lemma}
\begin{proof}
    Suppose that $g\in I_{K/\Q_2}$, $g\neq 1$ acts trivially on $D$, and let $F:=K^{\langle g,t \rangle}$. Then $D$ has trivial inertia action as a filtered $(\varphi,\Gal(K/F))$-module. Thus the corresponding Weil--Deligne representation $\rho_{WD}$ is unramified, so $E$ attains good reduction over $F\subsetneq L=K^{\langle t \rangle}$, contradicting the minimality of $L$. Since $\Lambda^2 D=K_0\{-1\}$ has trivial $I_{K/\Q_2}$-action, we must have $\det g=1$.
    
    Finally, from $\langle t\rangle=\Gal(K/L)$ it follows that $D_{cris,L}^{*}(V_2(E))\simeq D_{cris,K}^{*}(V_2(E))^{\Gal(K/L)}\simeq D_{cris,K}^{*}(V_2(E))^{\langle t\rangle}=D_0$. Moreover, we have $D_{cris,L}^{*}(V_2(E))\simeq D_{cris,L}(H^1_{et}(E,\Q_2))\simeq H_{cris}^{1}(\tilde{E}/\ma Z_2)[1/2]$, where $\tilde{E}$ is the reduced curve. By \cite[Theorem~1]{KatzMessing1974}, the characteristic polynomial of the crystalline Frobenius acting on $H_{cris}^{1}(\tilde{E}/\Q_2)$ is the same as the characteristic polynomial of the relative $q$-th power Frobenius acting on $H^1_{et}(\tilde{E}/\overline{\F}_2,\Q_l)$. Since $L$ is totally ramified, $q=p=2$, and the characteristic polynomial is known to be $X^2-a_2X+2$ with $a_2\in \ma Z$. In particular, $a_2^2-8\neq 0$.
\end{proof}

\begin{lemma}\label{L: pot supersingular}
    Let $E/\Q_2$ be an elliptic curve with potentially good reduction and let $e$ be its semistability defect. If $e\geq 3$, then $E$ is potentially supersingular.
\end{lemma}
\begin{proof}
    See \cite[Theorem 3.2]{Dokchitser2015}.
\end{proof}

\begin{lemma}\label{L: invariant subspaces}
    Let $D$ be an $\upgmt$ that satisfies ${\bf(A.2)}$ for $K/L$. Then $D$ has no nonzero proper $K_0$-subspaces stable under $\varphi$. For any $\Gal(K/\Q_2)$-stable line $R$ in $D_K$ the filtration $\Fil^{i}D_K=D_K$ for $i\leq 0$, $\Fil^1D_K=R$ and $\Fil^{i}D_K=0$ for $i\geq 2$ is admissible. In particular, the admissible filtrations are parametrize by $\PP^1(\Q_2)$.  
\end{lemma}
\begin{proof}
    By Lemmas~\ref{L: adapted basis} and~\ref{L: a_2=0}, $M(\varphi)=\tbt{0}{-2}{1}{0}$. If $\varphi(v)=\lambda v$, semilinearity gives $\varphi^{2}(v)=\operatorname{Nm}(\lambda) v$. However, $\varphi^2=-2Id$, so $\operatorname{Nm}(\lambda)=-2$, which is impossible since $K_0$ is unramified. If $D$ has Hodge--Tate weights $(0,1)$, then, since $\dim_{K_0}(D)=2$, we have $t_{H}(D)=1$; the description of $\varphi$ gives $t_N(D)=1$. Since there are no invariant subspaces then the condition $t_N(\tilde{D})\geq t_H(\tilde{D})$ on subobjects is automatically satisfied. Moreover, with the filtration of the statement $t_N(D)=t_H(D)=1$ and so the filtration is admissible. We have shown that the admissible filtrations are parametrize by lines in $(D\otimes_{K_0}K)^{\Gal(K/\Q_2)}$ which by Galois descent is a $2$-dimensional $\Q_2$ vector space thus $\PP((D\otimes_{K_0}K)^{\Gal(K/\Q_2)})\simeq \PP^{1}(\Q_2)$.
\end{proof}

The above lemma concerns admissible filtration not isomorphisms class of admissible filtrations. To shows that each parameter in $\PP^{1}(\Q_2)$ defines a nonisomorphic $\fpgmt$ we will use the following lemma.

\begin{lemma}\label{L: unfiltered automorphism}
   Let $D$ be an $\upgmt$ satisfying ${\bf(A)}$. In the basis of Lemma~\ref{L: adapted basis}, every automorphism $f$ of the unfiltered module has $M(f)=c_0\Id+c_1M(\varphi)$ with $c_0,c_1\in\Q_2$. Suppose moreover that some $g\in I_{K/\Q_2}$ has nonscalar matrix and $\cl L(\Id,M(g))\cap\cl L(\Id,M(\varphi))=\cl L(\Id)$,
where $\cl L$ denotes $K_0$-linear span. Then $f$ is scalar. 
\end{lemma}
\begin{proof}
    Since $f$ commutes with $t$ and $M(t)=\Id$, $\sigma(M(f))=M(f)$, so $M(f)\in M_2(\Q_2)$. Commutation with $\varphi$ therefore becomes $M(f)M(\varphi)=M(\varphi)M(f)$ since $M(\varphi)$ is not scalar then $M(f)\in\cl L(\Id,M(\varphi))$ wich proves the first assertion. Similarly, commutation with the linear operator $g$ implies $M(f)\in\cl L(\Id,M(g))$. By the hypothesis on the intersection of these spaces $f$ must be scalar.
\end{proof}

\begin{remark}
    When a $\fpgmt$ satisfies the condition of Lemma~\ref{L: invariant subspaces} it has a $\PP^{1}(\Q_2)$ of admissible filtrations, if we can also apply Lemma~\ref{L: unfiltered automorphism} any unfiltered automorphism is scalar and thus the filtration for each parameter in $\PP^{1}(\Q_2)$ defines a nonisomorphic $\fpgmt$. This will be the situation in most of the cases below, thus, even though we will only be able to give a description in terms of equations for the parameter, and not the parametrization itself the last entry of Table~\ref{Ta: Main} is justified. 
\end{remark}

The following is well known.
\begin{lemma}\label{L: Conditions on DVpE}
    Let $E/\Q_2$ be an elliptic curve with potentially good reduction. Then:
 \begin{itemize}
        \item $D_{pcris}^*(V_2(E))$ can be represented by $D$, a $\fpgmt$ in which $K$ is the Galois closure of one of the fields $L$ in Table~\ref{Ta: InTypes}.
        \item $P_{char}(\varphi\vert_{D_{0}})=X^2-a_2X+2$ with $a_2\in \Z$ and $\vert a_2\vert \leq 2\sqrt{2}$.
        \item $\operatorname{Tr}(\rho_{WD}(D)(g))\in \Q$ for all $g\in W_{\Q_2}$
        \item $\Lambda^2 D=K_0\{-1\}$
        \item $\dim_K\Fil^0D_K=2$, $\dim_K\Fil^1D_K=1$, $\dim_K\Fil^2D_K=0$
    \end{itemize}
\end{lemma}
\begin{proof}
    The first statement follows from Lemma~\ref{L: InTypes} and the second from Lemma~\ref{L: curves satisfy A} together with the Hasse-Weil bound. Since the Weil--Deligne representations of $E$ are independent of $\ell$ and traces are known to be rational for $\ell\neq p$, the third point follows. The fourth statement follows by applying $D_{cris,K}^{*}$ to the equality $\det \rho_{E,p^{\infty}}=\chi_p$. Finally, the last statement follows from the fact that the Hodge--Tate weights of $\rho_{E,p^{\infty}}$ are $\{0,1\}$.
\end{proof}

\begin{remark}\label{Rmk: Avoid rational traces hypothesis}
    The last four conditions of Lemma~\ref{L: Conditions on DVpE} are the usual necessary conditions for elliptic curves. For $p\geq 5$, none is superfluous: Volkov explicitly uses all of them to exclude objects from the classification. However we do not need to impose the condition $\operatorname{Tr}(\rho_{WD}(D)(g))\in \Q$; it follows once we impose all the others, see~\cite{MAGMAFiles}. We believe that this also holds for $p=3$, but we have not repeated all the computations of~\cite{bosco20233adicrepresentationsarisingelliptic} to verify it. Volkov uses this condition to exclude certain combinations of fields $L$ and traces $a_p$. For $p=2$ (and perhaps $p=3$), this is likely unnecessary because of the small number of traces allowed by the Hasse--Weil bound.   
\end{remark}

\section{Crystalline $e=1$}

\subsection{Ordinary}

Let $a_2\in\{-1,1\}$, let $u\in \Z_2^{\times}$ be the unique root of
$X^2-a_2X+2$ with valuation $0$, and let $\alpha\in\{0,1\}$. We define the $\fpgmt$
$D_c(1,\pm1,\alpha)$ as
\begin{itemize}
    \item $D:=\Q_2e_1\oplus\Q_2e_2$
    \item $M(\varphi)=\begin{pmatrix}
                  u & 0       \\
                  0 & 2u^{-1} \\
              \end{pmatrix}$
    \item $\Fil^1D:=(\alpha e_1 +e_2)\Q_2$
\end{itemize}
Note that different pairs $(a_2,\alpha)\in\{-1,1\}\times\{0,1\}$ produce nonisomorphic $D_c(1,\pm1 ,\alpha)$.

\begin{proposition}\label{P: e=1ordinary}
    Let $E/\Q_2$ be an elliptic curve with good ordinary reduction and trace of Frobenius $a_2$. Then $D_{cris}^{*}(V_2(E))\simeq D_c(1,a_2,\alpha)$ for a uniquely determined $\alpha$.

\end{proposition}

\begin{proof}
    In this case, $K_0=\Q_2$, so $D$ is a $\Q_2$-vector space and $\varphi$ is a linear endomorphism with characteristic polynomial $P_{char}(\varphi)=X^2-a_2X+2$. Since $a_2$ is a unit, by Hensel's lemma this polynomial has two distinct roots $u,2u^{-1}$ for $u\in \Z_{2}^{\times}$. Thus there is a basis in which the matrix of $\varphi$ is

    $$M(\varphi)=\begin{pmatrix}
            u & 0       \\
            0 & 2u^{-1} \\
        \end{pmatrix}.$$

    Let $D_1,D_2$ be the subspaces generated by $e_1$ and $e_2$, respectively. These
    subspaces are stable under $\varphi$ and therefore enter the weak
    admissibility condition. If $\Fil^{1}(D)=D_1$, we have $t_N(D_1)=0<t_H(D_1)=1$, which
    violates the weak admissibility of $D$. It follows that any filtration must be
    of the form $\Fil^{1}(D)=\langle \alpha e_1+e_2\rangle$ with $\alpha\in \Q_2$. Since any
    diagonal matrix commutes with $\varphi$, all the filtrations with $\alpha\neq
        0$ are equivalent, so it suffices to study the weak admissibility of
    $\Fil^{1}(D)=D_2$ and $\Fil^{1}(D)=\langle e_1+e_2\rangle$.

    If $\Fil^{1}(D)=D_2$, i.e., $\alpha=0$, then $t_N(D_1)=t_H(D_1)=0$ and $t_N(D_2)=t_H(D_2)=1$. It
    follows that $D$ is weakly admissible and also that $D_1,D_2$ are weakly
    admissible. Moreover, it is easy to check that $D=D_1\oplus D_2$ as $\fpgmt$s. Since the Hodge--Tate weights of $D_1$ and $D_2$ are $0$ and $1$, respectively, thus we have $D_{pcris}^{*}(\eta_1)=D_1$ and $D_{pcris}^{*}(\chi_2\otimes\eta_2)=D_2$, where $\eta_i$ are unramified characters. Since the determinant of $V_2(E)$ is the cyclotomic character, $\eta_2=\eta_1^{-1}$. Since $D_{pcris}^{*}$ is an antiequivalence of categories, it follows that
    $V_2(E)=\eta_1 \oplus \eta_1^{-1} \chi_2$, where $\eta_1$ is an unramified character
    and $\chi_2$ is the $2$-adic cyclotomic
    character.

    If $\Fil^{1}(D)=\langle e_1+e_2\rangle$, then $t_N(D_1)=t_H(D_1)=0$ and $t_N(D_2)=1>t_H(D_2)=0$. In
    this case $D$ and $D_1$ are weakly admissible but $D_2$ is not. It follows that
    $D_1$ is a subobject in $MF_{\Q_p}^{\varphi,w.a.}$ while $D_2$ is a subobject in
    $MF_{\Q_2}^{\varphi}$; thus it does not correspond to a $p$-adic representation. However, since the category
    $MF_{\Q_2}^{\varphi,w.a.}$ is abelian, the subobject $D_1$ gives rise to a nonsplit short
    exact sequence

    $$ 0\longrightarrow D_1\longrightarrow D\longrightarrow \tilde{D}_2\longrightarrow 0
    $$
    where $\tilde{D}_2\simeq D/D_1$. It is easy to check that the Hodge--Tate weight
    of $\tilde{D_2}$ is $1$ and so it corresponds via $D_{pcris}^{*}$ to a character
    which, by the condition on the determinant, is the product of the unramified character $\eta_1$ and $\chi_2$. Since
    $D_{pcris}^{*}$ is contravariant, $V_2(E)$ fits into the nonsplit short
    exact sequence $$ 0\longrightarrow \chi_2\otimes\eta_1^{-1} \longrightarrow V_2(E)
        \longrightarrow \eta_1 \longrightarrow 0 $$

    Finally, $D_{c}(1,a_2,\alpha)$ and $D_{c}(1,\tilde{a_2},\beta)$ can only be isomorphic if $a_2=\tilde{a_2}$; otherwise the Frobenius operators have different eigenvalues, and so there is no isomorphism commuting with them. Moreover, if $a_2=\tilde{a_2}$, since $\varphi$ has different eigenvalues, an automorphism commuting with it must be of the form $c_0\Id+c_1\varphi$. This matrix is diagonal, so it cannot send the line $e_1+e_2$ to the line $e_2$.
 \end{proof}

\subsection{Supersingular}
Let $a_2\in\{-2,0,2\}$. We define $D_c(1,a_2)$ as
\begin{itemize}
    \item $D:=\Q_2e_1\oplus\Q_2e_2$
    \item $M(\varphi):=\begin{pmatrix}
                  0 & -2  \\
                  1 & a_2 \\
              \end{pmatrix}$
    \item $\Fil^1D:=e_1\Q_2$
\end{itemize}

\begin{proposition}\label{P: e=1supersingular}
    Let $E/\Q_2$ be an elliptic curve with good supersingular reduction and trace of Frobenius $a_2$. Then $D_{pcris}^{*}(V_2(E))\simeq D_c(1,a_2)$.

\end{proposition}
\begin{proof}
    Let $D=D_{pcris}^{*}(V_2(E))$. Again, since the curve has good reduction, $\varphi$ is a linear map with $P_{char}(\varphi)=X^2-a_2X+2$. Since $a_2\equiv0 \pmod2$, it follows that $P_{char}$ is irreducible, so there is no proper subspace of $D$ stable under $\varphi$. Thus any filtration in $D$ is automatically weakly admissible. Let $e_1$ be a nonzero vector of $\Fil^1(D)$ and $e_2=\varphi(e_1)$. In this basis, the matrix of $\varphi$ has the form
    $$ M(\varphi)=\begin{pmatrix}
            0 & -2  \\
            1 & a_2 \\
        \end{pmatrix}$$
    which finishes the proof.

\end{proof}

\begin{remark}
    The choice of $e_1$ as the vector spanning $\Fil^{1}D_K$ only makes sense when $K=K_0$, in which case $D_K=D$ and the filtration is by subspaces of $D$. For any ramified extension $K/\Q_2$, the filtration is by subspaces of $D_K=D\otimes_{K_0} K$, so it does not make sense to choose $e_1\in \Fil^1(D_K)$ as one of the vectors in a basis for $D$.

\end{remark}

The following is well known, but we state it here for completeness.

\begin{theorem}\label{T: realization e=1}
    Let $D\in\{D_c(1,0),D_{c}(1,\pm2),D_{c}(1,\pm1,0),D_{c}(1,\pm1,1)\}$. Then there exists an elliptic curve $E/\Q_2$ such that $D_{pcris}^{*}(V_2(E))=D$.
 
\end{theorem}
\begin{proof}
 For supersingular $a_2$, the assertion follows because there is only one $\fpgmt$ for each trace and there exist elliptic curves $E/\Q_2$ with $a_2\in\{-2,0,2\}$. In the ordinary case, it suffices to produce two elliptic curves with $a_2=1$: a curve $E_s$ for which $V_2(E_s)$ is a sum of two characters ($\alpha=0$), and a curve $E_{ns}$ for which $V_2(E_{ns})$ is reducible and nonsplit ($\alpha=1$). Twisting $V_2(E)$ by the unramified quadratic character preserves whether it is split and transforms curves with $a_2=1$ into curves with $a_2=-1$. Thus the twists of $E_s$ and $E_{ns}$ provide the desired curves for $a_2=-1$. According to \cite[Chapter~IV, Appendix, \S A.2.4]{Serre1989} for $E_s$, we can choose an elliptic curve with $\CM$ by a field in which $2$ splits, and for $E_{ns}$ a curve without $\CM$. Such curves exist; for example, one can choose $$ \href{https://www.lmfdb.org/EllipticCurve/Q/49/a/4}{E_s}:y^2+xy=x^3-x^2-2x-1\quad\quad\quad \href{https://www.lmfdb.org/EllipticCurve/Q/3897/c/1}{E_{ns}}:y^2+xy=x^3-x^2-27$$
 \end{proof}

\section{Semistability defect $e=2$}

One might expect the case $e=2$ to be more difficult when $p=2$. However, the
difficulties in the wild cases arise from the existence of several extensions over which the
curve attains good reduction and from noncyclic Galois groups,
such as $C_3\rtimes C_4$ for $p=3$ or $Q_8$ for $p=2$. When
$e=2$ and $p=2$, the Galois groups are always $C_2$, and the fields of good
reduction can be chosen from $\Q_2(\sqrt{2}),\Q_2(\sqrt{-1})$ and $\Q_2(\sqrt{-2})$.
These cases therefore reduce to twists of the crystalline cases.

\begin{lemma}\label{L: D eta}
    Let $\Q_4=\Q_2(\sqrt{5})$ be the unramified quadratic extension, let $\eta$ be the associated quadratic character and let $D_{\eta}:=D^*_{pcris}(\eta)$. Then:

    \begin{itemize}
        \item $D_{\eta}=\Q_2e_1$
        \item $\varphi=-Id$
        \item $\Fil^0D_{K}=e_1 K$ and $\Fil^1D_{K}=0$
    \end{itemize} 
\end{lemma}
\begin{proof}
    The unramified character $\eta$ is crystalline, so $D_{pcris}^{*}(\eta)$ can be represented by a filtered $(\varphi,\Gal(\Q_2/\Q_2))$-module, i.e., by an object of $MF_{\Q_2}^{\varphi,w.a.}$. The filtration follows from its Hodge--Tate weight $0$. The equality $\varphi=-Id$ follows from $D_{\eta}\otimes D_{\eta}=\Q_p\{0\}$, which gives $\varphi^2=1$, together with the fact that $\varphi=Id$ would imply that $\eta$ is trivial.
 
\end{proof}

\begin{lemma}\label{L: D varepsilon}
    Let $K=\Q_2(\sqrt{d})$ be a ramified quadratic extension, let $\varepsilon_d$ be the associated quadratic character, let $\langle x\rangle=\Gal(K/\Q_2)$ and let $D_{\varepsilon_d}:=D^*_{pcris}(\varepsilon_d)$. Then:

    \begin{itemize}
        \item $D_{\varepsilon_d}=\Q_2e_1$
        \item $\varphi=Id$
        \item $x=-Id$
        \item $\Fil^0D_{K}=e_1 K$ and $\Fil^1D_{K}=0$
    \end{itemize}
\end{lemma}

\begin{proof}
    Since $\varepsilon_d\vert_{G_K}$ is trivial, it follows that $D_{cris,K}^{*}(\varepsilon_{d})=\Q_2$ with $\varphi=Id$, $\Fil^{0}(D_K)=D_K$ and $\Fil^{1}(D_K)=0$. In order to compute $D_{pcris}^{*}(\varepsilon_{d})$, we need to furnish $D_{cris,K}^{*}(\varepsilon_{d})$ with a semilinear action of $\Gal(K/\Q_p)=\langle x\rangle$. Since $x$ is an element of inertia, it acts linearly on $D$; since $ord(x)=2$, either $x(e_1)=e_1$ or $x(e_1)=-e_1$. Because $\varepsilon_{d}$ is finite, its Hodge--Tate weight is $0$. Since the crystalline characters of weight $0$ are the unramified characters, it follows that $\varepsilon_{d}$ is not crystalline. Since $\dim_{\Q_p}(D_{cris,\Q_p}(\varepsilon_{d}))\leq \dim_{\Q_p}(\varepsilon_{d})=1$ with equality if and only if $\varepsilon_{d}$ is crystalline, it follows that $D_{cris,\Q_p}(\varepsilon_{d})=0$. Finally, we know that $D_{cris,K}^{*}(\varepsilon_{d})^{\Gal(K/\Q_2)}=D_{cris,\Q_p}=0$; thus the action of $x$ on $D$ cannot have any fixed vector. It follows that $x(e_1)=-e_1$, which completes the proof.

\end{proof}

\begin{proposition}\label{P: e=2}
    Let $E/\Q_2$ be an elliptic curve with semistability defect $e=2$ that attains good reduction over a quadratic extension $K/\Q_2$. Let $\varepsilon_{d}$ be the associated quadratic character, let $\langle x\rangle=\Gal(K/\Q_2)$ and let $D_{\varepsilon_{d}}=D_{pcris}^{*}(\varepsilon_{d})$.

    If $E$ has potentially ordinary reduction, then
    $D:=D_{pcris}^{*}(V_2(E))=D_c(1,a_2,\alpha)\otimes D_{\varepsilon_{d}}$ for some uniquely determined
    $\alpha\in\{0,1\}$; more explicitly:

    \begin{itemize}
        \item $D=\Q_2e_1\oplus\Q_2e_2$
        \item $M(\varphi)=\begin{pmatrix}
                      u & 0       \\
                      0 & 2u^{-1} \\
                  \end{pmatrix}$
        \item $M(x)=\begin{pmatrix}
                      -1 & 0  \\
                      0  & -1 \\
                  \end{pmatrix}$
        \item $\Fil^1D_K=(\alpha\tenso e_1 +1\tenso e_2)K$
    \end{itemize}

    If $E$ has potentially supersingular reduction, then
    $D:=D_{pcris}^{*}(V_2(E))=D_c(1,a_2)\otimes D_{\varepsilon_{d}}$; more explicitly:
    \begin{itemize}
        \item $D=\Q_2e_1\oplus\Q_2e_2$
        \item $M(\varphi)=\begin{pmatrix}
                      0 & -2  \\
                      1 & a_2 \\
                  \end{pmatrix}$
        \item $M(x)=\begin{pmatrix}
                      -1 & 0  \\
                      0  & -1 \\
                  \end{pmatrix}$
        \item $\Fil^1D_{K}:=e_1\cdot K$
    \end{itemize}

\end{proposition}
\begin{proof}
    Let $E$ be an elliptic curve with potentially good reduction and semistability defect $e=2$. It is well known that there exists a ramified quadratic extension $\Q_2(\sqrt{d})/\Q_2$ such that $E_d$, the quadratic twist of $E$ by $d$, has good reduction. Thus $D_{pcris}^{*}(V_2(E_d))$ must be among the modules already computed. Moreover, since $V_2(E_d)\otimes \varepsilon_{d} \simeq V_2(E)$ and $D_{pcris}^{*}$ is a tensor antiequivalence, we have $D_{pcris}^{*}(V_2(E))=D_{pcris}^{*}(V_2(E_d))\otimes D_{pcris}^{*}(\varepsilon_{d})$. The result then follows from Propositions~\ref{P: e=1ordinary} and~\ref{P: e=1supersingular}.

\end{proof}

\section{Semistability defect $e=3$}

For elliptic curves $E/\Q_2$, there is only one inertia type $\tau$ with $e=3$; see Lemma~\ref{L: InTypes}. An
elliptic curve $E$ with this type attains good reduction over the field $L=\Q_2(\sqrt[3]{2})$ whose Galois closure is $K=\Q_2(\sqrt[3]{2},\zeta_3)$ with
$\Gal(K/\Q_2)\simeq S_3\simeq D_3$. The maximal unramified subextension of $K/\Q_2$ is
the field $K_0=\Q_4=\Q_2(\zeta_3)$. We have the following presentation
$\Gal(K/\Q_2)=\langle x,t~|~x^3=1,t^2=1,xt=tx^{-1}\rangle$, where
$$x(\sqrt[3]{2})=\zeta_3\sqrt[3]{2},\quad x(\zeta_3)=\zeta_3,\quad
    t(\sqrt[3]{2})=\sqrt[3]{2},\quad t(\zeta_3)=\zeta_3^2.$$ In particular, the
fixed field of $x$ is $K_0$ and the fixed field of $t$ is $\Q_2(\sqrt[3]{2})$.
Thus $I_{K/\Q_2}=\langle x\rangle$.

We fix $\sqrt{-3}=2\zeta_3+1$ and a square root $\sqrt5\in\Q_4$. Throughout the paper, $\sqrt{-15}$ means the product $\sqrt5\sqrt{-3}$, which lies in $\Q_2$. For a projective parameter
$\alpha=[A:B]\in\PP^1(\Q_2)$, the notation
$K(\alpha e_1+c e_2)$ means $K(Ae_1+Bc e_2)$; in particular,
$\alpha=\infty$ means the line $Ke_1$.

\begin{defn}\label{D: D3}
    Let $\zeta_3$ be as above. We define the $\fpgmt~D_{pc}(3,0,\alpha)$ as follows:

    \begin{itemize}
        \item $D_{pc}(3,0,\alpha):=\Q_4e_1\oplus\Q_4e_2$
        \item $M(\varphi)=\begin{pmatrix}
                      0 & -2 \\
                      1 & 0  \\
                  \end{pmatrix}$
        \item $M(x)=\begin{pmatrix}
                      \zeta_3 & 0         \\
                      0       & \zeta_3^2 \\
                  \end{pmatrix}$
        \item $M(t)=\begin{pmatrix}
                      1 & 0 \\
                      0 & 1 \\
                  \end{pmatrix}$
        \item $\Fil^1D_{K}:=(\alpha\tenso e_1+(\sqrt[3]{2})^{2}\tenso e_2)K$ with $\alpha\in\PP^1(\ma Q_2)$
    \end{itemize}
\end{defn}

\begin{proposition}\label{P: classification D3}
    Let $L=\Q_2(\sqrt[3]{2})$ and $K=\Q_2(\sqrt[3]{2},\zeta_3)$. Let $D$ be a $\fpgmt$ satisfying ${\bf (A.2)}$ for $K/L$. Then $D\simeq D_{pc}(3,0,\alpha)$ as $\upgmt$s. If, moreover, $D$ has Hodge--Tate weights $(0,1)$, then $D\simeq D_{pc}(3,0,\alpha)$ as $\fpgmt$s for a unique $\alpha$.
\end{proposition}
\begin{proof}
    We work in the basis provided by Lemma~\ref{L: adapted basis}. Since $\zeta_3+\zeta_3^{-1}=-1$, Lemma~\ref{L: eigenvalues} shows that the minimal polynomial of $x$ acting on $D$ is $X^2+X+1$. In particular, the matrix describing the action of $x$ on $D$ has the form $$M(x)=\begin{pmatrix}
            \lambda & \mu'       \\
            \mu     & -1-\lambda \\
        \end{pmatrix}$$.

        By Lemma~\ref{L: a_2=0}, we have $a_2=0$. The condition $xt=tx^{-1}$ gives the equations $$
        t(\lambda)=-1-\lambda\quad, t(\mu)=-\mu\quad, t(\mu')=-\mu'\quad,
        t(-1-\lambda)=\lambda$$
        and the condition $\varphi x=x\varphi$ gives equations
        $$
 t(\lambda)=-1-\lambda,\qquad
\mu'=-2t(\mu),\qquad
t(\mu')=-2\mu.$$
    The equation $t(\lambda)=-1-\lambda$ implies
    $\lambda=\frac{-1}{2}+u_1\sqrt{-3}$ with $u_1\in\Q_2$. Similarly, $t(\mu)=-\mu$
    implies $\mu=u_2\sqrt{-3}$ with $u_2\in\Q_2$. From the equations for $\varphi$ we obtain $\mu'=2\mu=2u_2\sqrt{-3}$. Then, the condition $\det(x)=1$
    implies $u_2=\pm\frac{1}{2}\sqrt{\frac{1-4u_1^2}{2}}$. Thus the matrix of $x$ is of the form $$M(x)=\begin{pmatrix}
            \frac{-1}{2}+u_1\sqrt{-3} & 2u_2\sqrt{-3}             \\
            u_2\sqrt{-3}              & \frac{-1}{2}-u_1\sqrt{-3} \\
        \end{pmatrix}$$

    We seek a change of basis that commutes with $\varphi$ and
    sends this matrix to the one appearing in $D_{pc}(3,0,\alpha)$. The equations defining
    such a matrix can be solved using a Gröbner basis. The change of basis is given by the
    matrix $$ B=\begin{pmatrix}
            (u_1+\frac{1}{2}) & -2u_2               \\
            u_2             & (u_1+\frac{1}{2}) \\
        \end{pmatrix}~\text{if}~(u_1,u_2)\neq(-\frac{1}{2},0)\quad\text{and}\quad B=\tbt{0}{-2}{1}{0}~\text{otherwise}$$

    The two possible solutions for $u_2$ lead to
    isomorphic objects. See~\cite{MAGMAFiles} for a {\tt Magma} script that checks all
    these computations. This proves the unfiltered isomorphism.

    It only remains to determine the $\Gal(K/\Q_2)$-stable line corresponding to
    $\Fil^1(D_K)$. By Lemma~\ref{L: invariant subspaces}, any
    $\Gal(K/\Q_2)$-stable line gives rise to an admissible filtration. The line spanned by $1\otimes e_1$ is invariant under $\Gal(K/\Q_2)$. The remaining invariant lines can therefore be represented by an element of the
    form $\beta\otimes e_1+1\otimes e_2$ with $\beta\in K$. Then, the condition
    $t(v)=\gamma v$ reads $$ t(v)=t(\beta)\otimes e_1+1\otimes e_2=\gamma
        v=\gamma\beta\otimes e_1+\gamma\otimes e_2\iff t(\beta)=\beta\iff \beta\in
        K^{\langle t\rangle}= \Q_2(\sqrt[3]{2}),$$ and similarly the condition
    $x(v)=\gamma v$ reads $$ x(v)=\zeta_3x(\beta)\otimes e_1+\zeta_3^2 \otimes
        e_2=\gamma v=\gamma\beta\otimes e_1+\gamma\otimes e_2\iff
        x(\beta)=\zeta_{3}\beta.$$

    This last equation implies that
    $x((\sqrt[3]{2})^2\beta)=(\sqrt[3]{2}\zeta_3)^2x(\beta)=(\sqrt[3]{2})^2\beta$; thus $(\sqrt[3]{2})^2\beta\in K^{\langle x \rangle}=\Q_2(\zeta_3)$. Since
    $(\sqrt[3]{2})^2\beta\in \Q_2(\sqrt[3]{2})$, we have $(\sqrt[3]{2})^2\beta\in\Q_2$, so $\beta=\alpha(\sqrt[3]{2})^{-2}$ for some $\alpha\in \Q_2$. It follows that
    any $\Gal(K/\Q_2)$-stable line is spanned by $\alpha\tenso e_1+(\sqrt[3]{2})^{2}\tenso e_2$ with $\alpha\in\PP(\ma Q_2)$. It remains to show that, for $\alpha\neq\beta$, $D_{pc}(3,0,\alpha)$ and $D_{pc}(3,0,\beta)$ are not isomorphic. An isomorphism $f$ between $D_{pc}(3,0,\alpha)$ and $D_{pc}(3,0,\beta)$ is, in particular, an unfiltered automorphism of $D_{pc}(3,0,\alpha)$. By Lemma~\ref{L: unfiltered automorphism}, any such $f$ is scalar and can therefore be a filtered isomorphism only if $\alpha=\beta$.
    \end{proof}

    \begin{corollary}\label{Cor: classification for D3}
    Let $E/\Q_2$ have inertia type associated to $\Q_2(\sqrt[3]{2})$. Then $D_{pcris}^{*}(V_2(E))\simeq D_{pc}(3,0,\alpha)$ for a unique $\alpha$.
\end{corollary}
\begin{proof}
    This follows directly from Lemma~\ref{L: curves satisfy A} and Proposition~\ref{P: classification D3}.  
\end{proof}

\section{Semistability defect $e=4$}

Let $E/\Q_2$ be an elliptic curve with semistability defect $e=4$. By
Lemma~\ref{L: InTypes}, there is an extension $K/\Q_2$ over which $E$ attains good
reduction and such that $\Gal(K/\Q_2)\simeq C_4$ or $\Gal(K/\Q_2)\simeq D_4$.
In both cases, $I_{K/\Q_2}\simeq C_4$.

\subsection{Dihedral case} Assume that $L=L_{4,i}^{na}$, so $\Gal(K/\Q_2)\simeq D_4$. We use the presentation
$\Gal(K/\Q_2)\simeq D_4=\langle x,t~|~ x^4=1,~t^2=1,~txt=x^{-1} \rangle$, where
$I_{K/\Q_2}=\langle x\rangle$. With this description, $K^{\langle x
            \rangle}=\Q_4$ and we write $L=K^{\langle t \rangle}$.

\begin{defn}
    Let $K$ be one of the fields described above. We define $D_{pc}(4,0,i,\alpha)$ to be the following $\text{filtered}~(\varphi,\Gal(K/\Q_2))-\text{module}$:

    \begin{itemize}
        \item $D:=\Q_4e_1\oplus\Q_4e_2$
        \item $M(\varphi)=\begin{pmatrix}
                      0 & -2 \\
                      1 & 0  \\
                  \end{pmatrix}$
        \item $M(x)=\begin{pmatrix}
                      \sqrt{5}  & 2\sqrt{-3} \\
                      \sqrt{-3} & -\sqrt{5}  \\
                  \end{pmatrix}$
        \item $M(t)=\begin{pmatrix}
                      1 & 0 \\
                      0 & 1 \\
                  \end{pmatrix}$
        \item $\Fil^1D_{K}:=(\alpha\tenso e_1 +1\tenso e_2)K$ for $\alpha\in \Q_2(\sqrt{-2})$ such that $x(\alpha)=\frac{\alpha\sqrt{5}+2\sqrt{-3}}{\alpha\sqrt{-3}-\sqrt{5}}$
    \end{itemize}
\end{defn}

\begin{proposition}\label{P: classification D4}
    Let $L=L_{4,i}^{na}$ be an extension in Table~\ref{Ta: InTypes}. Let $D$ be a $\fpgmt$ satisfying ${\bf (A.2)}$ for $K/L$. Then $D\simeq D_{pc}(4,0,i,\alpha)$ as an $\upgmt$. If, moreover, $D$ has Hodge--Tate weights $(0,1)$, then $D\simeq D_{pc}(4,0,i,\alpha)$ as a $\fpgmt$ for a unique $\alpha$.
\end{proposition}

\begin{proof}
    As in the previous cases, we work in the basis provided by~\ref{L: adapted basis}. Since $\zeta_4+\zeta_4^{-1}=0$, Lemma~\ref{L: eigenvalues} implies that the minimal polynomial of $x$ acting on $D$ is $X^2+1$. Thus, in this basis, $$M(x)=\begin{pmatrix}
            \lambda & \mu'     \\
            \mu     & -\lambda \\
        \end{pmatrix}\quad M(t)=\begin{pmatrix}
            1 & 0 \\
            0 & 1 \\
        \end{pmatrix}$$

        By Lemma~\ref{L: a_2=0}, we have $a_2=0$. From the relation $txt=x^{-1}$, we obtain $$t(\lambda)=-\lambda \quad
        t(\mu)=-\mu\quad t(\mu')=-\mu'$$

    It follows that $\lambda=u_1\sqrt{5}$, $\mu=u_2\sqrt{-3}$ and
    $\mu'=u_3\sqrt{-3}$ with $u_i\in\Q_2$. The matrix of $x$ on $D$ therefore has the form $$ M(x)=\tbt{u_1\sqrt{5}}{2u_2\sqrt{-3}}{u_2\sqrt{-3}}{-u_1\sqrt{5}} $$

    The matrix $$B=\begin{pmatrix}
            5(u_1+1)          & -2\sqrt{-15}(u_2-1) \\
            \sqrt{-15}(u_2-1) & 5(u_1+1)            \\
        \end{pmatrix}~\text{if}~(u_1,u_2)\neq(-1,1)\quad\text{and}\quad B=\tbt{2\sqrt{-15}}{-10}{5}{2\sqrt{-15}}~\text{otherwise}$$ defines a change of basis that commutes with $\varphi$ and $t$ and sends the matrix of $x$ to the one in $D_{pc}(4,0,i,\alpha)$. This proves the unfiltered isomorphism.

        It remains to check that the filtration has the prescribed form. By Lemma~\ref{L: invariant subspaces}, any $\Gal(K/\Q_2)$-stable line defines an admissible filtration. Since the line generated by $e_1$ is not stable under $\Gal(K/\Q_2)$, we can
    assume that $\Fil^{1}D_K$ is generated by $\alpha\otimes e_1+1\otimes e_2$. Such
    a line is stable under $t$ if $t(\alpha)=\alpha$; thus $\alpha\in K^{\langle t \rangle}=L$. Moreover, stability under $x$ requires

    $$ \alpha=\frac{x(\alpha)\sqrt{5}+2\sqrt{-3}}{x(\alpha)\sqrt{-3}-\sqrt{5}}\iff x(\alpha)=\frac{\alpha\sqrt{5}+2\sqrt{-3}}{\alpha\sqrt{-3}-\sqrt{5}} $$

    Applying $x$ to the last equation gives $x^2(\alpha)=\alpha$; thus
    $\alpha$ is fixed by $x^2$ and so it belongs to the quadratic field $K^{\langle  x^2,t\rangle}$. For all $L^{na}_{4,i}$, this field is $\Q_2(\sqrt{-2})$ (it is the only quadratic subfield of $L^{na}_{4,i}=K^{\langle t \rangle}$). Thus $\alpha$ is an element of $\Q_2(\sqrt{-2})$ that satisfies
    $x(\alpha)=\frac{\alpha\sqrt{5}+2\sqrt{-3}}{\alpha\sqrt{-3}-\sqrt{5}}$, as
    desired.

    Moreover, each such $\alpha$ gives rise to a nonisomorphic $\fpgmt$. This follows, as in the case of $D_3$, from Lemma~\ref{L: unfiltered automorphism}.
\end{proof}

\begin{corollary}\label{Cor: classification for D4}
    Let $E/\Q_2$ have inertia type associated with $L_{4,i}^{na}$. Then $D_{pcris}^{*}(V_2(E))\simeq D_{pc}(4,0,i,\alpha)$ for a unique $\alpha$.
\end{corollary}
\begin{proof}
    This follows directly from Lemma~\ref{L: curves satisfy A} and Proposition~\ref{P: classification D4}.  
\end{proof}

\subsection{Cyclic case}

Suppose now that $L=L^{ab}_{4,i}$ is as defined in Table~\ref{Ta: InTypes}. Then $\Gal(K/\Q_2)=\langle x\rangle\simeq C_4$.

Since $\Gal(K/\Q_2)\simeq C_4$, these fields contain a unique quadratic
subextension. In this particular case, $K^{\langle x^2
    \rangle}=\Q_2(\sqrt{2})$.

The cyclic case differs from the nonabelian cases: the supersingular trace $a=0$ is not possible. This is because the action of inertia factors through the automorphisms of the reduced curve. Since the extensions are totally ramified, the residue field must be $\F_2$ and the only elliptic curve over $\F_2$ with $a_2=0$ has $Aut_{\F_2}(\tilde{E})=C_2$; thus inertia isomorphic to $C_4$ is not possible. See $\S\ref{S: explicit realization}$ for more details.

\begin{defn}
    Let $K=L_{4,i}^{ab}$, let $a\in\{-2,2\}$, and identify the signs
$+$ and $-$ with $\varepsilon=1$ and $\varepsilon=-1$, respectively.
Define $D_{pc}(4,a_2,i,\varepsilon)$ by $D=\Q_2e_1\oplus\Q_2e_2$ and
$$M(\varphi):=\begin{pmatrix}0&-2\\1&a\end{pmatrix},
 \qquad M(x):=\varepsilon\left(\Id-\frac a2M(\varphi)\right),
 \qquad \Fil^1D_K=K(\sqrt2 e_1+e_2).$$
This gives the four modules $D_{pc}(4,\pm2,i,\pm)$.
\end{defn}

\begin{proposition}\label{P: classification C4}
       Let $L=L_{4,i}^{ab}$ be an extension in Table~\ref{Ta: InTypes}. Let $D$ be a $\fpgmt$ satisfying ${\bf (A.1)}$ for $K/L$ with $a_2\in\{-2,0,2\}$. Then $D\simeq D_{pc}(4,a_2,i,\varepsilon)$ as an $\upgmt$ for some $a_2\in\{-2,2\}$ and $\varepsilon\in\{-,+\}$. If, moreover, $D$ has Hodge--Tate weights $(0,1)$, then they are also isomorphic as $\fpgmt$s.
\end{proposition}

\begin{proof}
    We work in the basis given by Lemma~\ref{L: adapted basis}. In this case, $t=Id$ and $\varphi=\varphi_0$ is linear. By Lemma~\ref{L: eigenvalues}, since $\ord(x)=4$, we have $\operatorname{Tr}(x)=0$. The matrix of $x$ in this basis is of the form

    $$M(x)=\begin{pmatrix}
            \lambda & \mu'     \\
            \mu     & -\lambda \\
        \end{pmatrix}$$

    Since $\varphi$ is not scalar, $\varphi\circ x=x\circ \varphi$ implies that $$M(x)=c_0Id+c_1M(\varphi)=\tbt{c_0}{-2c_1}{c_1}{c_0+a_2c_1}$$.

    This, together with $\Tr(x)=0$, implies $2c_0+a_2c_1=0$. In particular, $a_2\neq0$; otherwise $c_0=0$ and $\det M(x)=2c_1^2=1$ has no solution in $\Q_2$. Thus we can assume $a_2\neq 0$. Then $c_1=\frac{-2c_0}{a_2}$. Imposing the additional condition $\det M(x)=1$, we find that
    if $a_2=2$, the possible solutions are
    \[ M^{+}(x)=Id-M(\varphi)=\tbt{1}{2}{-1}{-1}\quad M^{-}(x)=-Id+M(\varphi)=\tbt{-1}{-2}{1}{1}
    \]
    and if $a_2=-2$, the possible solutions are
    \[ M^{+}(x)=Id+M(\varphi)=\tbt{1}{-2}{1}{-1}\quad\quad M(x)^{-}=-Id-M(\varphi)=\tbt{-1}{2}{-1}{1}
    \]

    We now focus on the case $a_2=2$, since the case $a_2=-2$ is analogous. The two solutions for $a_2=2$ are nonisomorphic, as an isomorphism would be
    given by a matrix $B$ satisfying $BM(\varphi)=M(\varphi) B$ and
    $B(Id-M(\varphi))=(-Id+M(\varphi))B$, whose only solution is $B=0$. Thus $M^{+}$ and $M^{-}$ give rise to $2$ nonisomorphic unfiltered $(\varphi,\Gal(K/\Q_2))$-modules that will correspond to $D_{pc}(4,2,i,\pm)$. This proves the unfiltered isomorphism.

    We now show that each of these modules admits exactly one isomorphism class of admissible filtrations.
    
    Since the minimal polynomial of $\varphi$ is irreducible over $\Q_2$, the operator has no invariant subspaces and so any $\Gal(K/\Q_2)$-stable line in $D_K$ defines an admissible $\Fil^1 D_K$. Clearly, the line $Ke_1$ is not stable under $x$, so we can assume that $\Fil^{1}(D_K)$ is generated by $\alpha\otimes e_1+1\otimes e_2$ for $\alpha\in K$. Stability under $x$ is then equivalent to the equations $$\alpha=\frac{x(\alpha)+2}{-x(\alpha)-1}~\quad \text{for}~M^{+}(x)\quad\text{and}\quad
        \alpha=\frac{-x(\alpha)-2}{x(\alpha)+1}\quad \text{for}~M^{-}(x)$$

    Both equations are equivalent to $x(\alpha)=\frac{\alpha+2}{-\alpha-1}$. Applying $x$ to the above equations, we deduce that $x^2(\alpha)=\alpha$; thus
    $\alpha$ lies in the unique quadratic subextension of $K$, namely $K^{\langle
                x^2\rangle}=\Q_2(\sqrt{2})$.
    In the case $M^{+}(x)$, given two $\Gal(K/\Q_2)$-stable filtrations $\alpha\otimes e_1+ 1\otimes e_2$ and $\beta\otimes e_1+ 1\otimes e_2$ with $\alpha\neq \beta$, we can consider the map $f=\frac{\alpha\beta+2\beta+2}{\alpha-\beta}Id+\varphi$. A computation shows that it commutes
    with $\varphi$ and $x$, that $\frac{\alpha\beta+2\beta+2}{\alpha-\beta}\in \Q_2$ and that $f\otimes Id$ sends the one-dimensional $K$-subspace
    of $D_K$ generated by $\alpha\otimes e_1+ 1\otimes e_2$ to the one generated by $\beta\otimes e_1+ 1\otimes e_2$, thereby proving that all such filtrations are equivalent. A similar choice of $f$ shows that all filtrations for $M^{-}(x)$ are also equivalent. Thus it is enough to compute one $\alpha$ that
    satisfies $x(\alpha)=\frac{\alpha+2}{-\alpha-1}$. We can rearrange this
    equation as $(x(\alpha)+1)(\alpha+1)=-1$ or $Nm(\alpha+1)=-1$. Thus we look for an element $1+\alpha$ in 
    $K^{\langle x^2 \rangle}=\Q_2(\sqrt{2})$ of norm $-1$. We can choose $1+\alpha=1+\sqrt{2}$; thus $\alpha=\sqrt{2}$. This finishes the proof for $a_2=2$. The proof for $a_2=-2$ is completely analogous.
\end{proof}

\begin{corollary}\label{Cor: classification for C4}
    Let $E/\Q_2$ have inertia type associated with $L_{4,i}^{ab}$. Then $D_{pcris}^{*}(V_2(E))\simeq D_{pc}(4,a_2,i,\varepsilon)$ for uniquely determined $a_2\in\{-2,2\}$ and $\varepsilon\in\{-,+\}$.
\end{corollary}
\begin{proof}
    This follows directly from Lemma~\ref{L: curves satisfy A} and Proposition~\ref{P: classification C4}.  
\end{proof}

\section{Semistability defect $e=6$}

Let $E/\Q_2$ be an elliptic curve with semistability defect $e=6$. By
Lemma~\ref{L: InTypes}, any such curve attains good reduction over $L=L_{6,i}$ for some $i$. Thus, if $K/\Q_2$ is the Galois closure, then $\Gal(K/\Q_2)\simeq D_{6}$. This group can be described as the
semidirect product $\Gal(K/\Q_2)\simeq I_{K/\Q_2}\rtimes \langle t \rangle$.
\[\Gal(K/\Q_2)=\langle x\rangle\rtimes\langle t \rangle=\langle x,t~|~x^6=1,~t^2=1,~txt=x^{-1}\rangle,\] where $x$ is a generator of $I_{K/\Q_2}\simeq C_6$ and $t$ is some Frobenius for
the extension $K/\Q_2$. We choose $\zeta_6=-\zeta_3^2$, so that
$\zeta_6^2=\zeta_3$. All these fields contain the field $\Q_2(\sqrt[3]{2},\zeta_3)$, and we can always choose $x,t$ so that $x(\sqrt[3]{2})=\zeta_3\sqrt[3]{2},~x(\zeta_6)=\zeta_6,~t(\sqrt[3]{2})=\sqrt[3]{2}$ and $t(\zeta_3)=\zeta_3^2$; that is, $x,t$ can be chosen to be lifts of the elements $x,t$ in the description of $D_{pc}(3,0,\alpha)$. In particular, $K^{\langle x^3\rangle}=\Q_2(\sqrt[3]{2},\zeta_6)$.

\begin{defn}
    Let $K$ be as described above. We define the $\fpgm~D_{pc}(6,0,i,\alpha)$ by:

    \begin{itemize}
        \item $D:=\Q_4e_1\oplus\Q_4e_2$
        \item $M(\varphi)=\begin{pmatrix}
                      0 & -2 \\
                      1 & 0  \\
                  \end{pmatrix}$
        \item $M(x)=\begin{pmatrix}
                      \zeta_6 & 0            \\
                      0       & \zeta_6^{-1} \\
                  \end{pmatrix}$

        \item $M(t)=\begin{pmatrix}
                      1 & 0 \\
                      0 & 1 \\
                  \end{pmatrix}$
        \item $\Fil^1D_{K}:=(\alpha\tenso e_1+\sqrt[3]{2}\tenso e_2)K$ with $\alpha\in\PP^1(\ma Q_2)$.
    \end{itemize}
\end{defn}

We reduce this case to the case $e=3$. The following lemma contains
the main calculation and is stated separately for its independent interest.

\begin{lemma}\label{L: Twist computation D6}
    Let $K/\Q_2$ be one of the above fields, let $F:=\Q_2(\sqrt{d})$ be the fixed field of $\langle x^2,t \rangle$ and let $\varepsilon_{d}$ be the quadratic character associated with $F$. Let $\tilde{K}$ be the fixed field of $\langle x^3\rangle$ and $D_{pc}(3,0,\alpha)$ the filtered $(\varphi,\Gal(\tilde{K}/\Q_2))$-module described in Definition~\ref{D: D3}. Then $D_{pc}(3,0,\alpha)\otimes D_{\varepsilon_{d}}\simeq D_{pc}(6,0,i,-4\alpha^{-1})$ as a $\fpgm$.
\end{lemma}

\begin{proof}
    To compute the tensor product, we must regard $D_{pc}(3,0,\alpha)$ and $D_{\varepsilon_{d}}$ as filtered $(\varphi, \Gal(K/\Q_2))$-modules as explained in \S2.

    The element $x\in\Gal(K/\Q_2)$ is a lift of $x\in\Gal(\tilde{K}/\Q_2)$ appearing in Definition~\ref{D: D3} and $x\in\Gal(F/\Q_2)$ in Lemma~\ref{L: D varepsilon}. These lifts make the action of $\Gal(K/\Q_2)$ explicit on
    $D_{pc}(3,0,\alpha)$ and on $D_{\varepsilon_{d}}$.

    The base change of $D_{pc}(3,0,\alpha)$ to $K$ is
    \begin{itemize}
        \item $D_{pc}(3,0,\alpha):=\Q_4e_1\oplus\Q_4e_2$
        \item $M(\varphi)=\begin{pmatrix}
                      0 & -2 \\
                      1 & 0  \\
                  \end{pmatrix}$
        \item $M(x)=\begin{pmatrix}
                      \zeta_3 & 0         \\
                      0       & \zeta_3^2 \\
                  \end{pmatrix}$
        \item $M(t)=\begin{pmatrix}
                      1 & 0 \\
                      0 & 1 \\
                  \end{pmatrix}$
        \item $\Fil^1D_{K}:=(\alpha\tenso e_1+\sqrt[3]{2}^2\tenso e_2)K$ with $\alpha\in\PP^1(\ma Q_2)$
    \end{itemize}

    The difference from the description as a filtered
    $(\varphi,\Gal(\tilde{K}/\Q_2))$-module is that now the group acting is $D_6$
    but the action factors through $D_3\simeq\Gal(\tilde{K}/\Q_2)$, reflecting the
    fact that the corresponding representation was already crystalline over
    $\tilde{K}$. For $D_{\varepsilon_{d}}$, the difference is more apparent. The base change of
    $D_{\varepsilon_{d}}$ to $K$ is as follows:

    \begin{itemize}
        \item $D_{\varepsilon_{d}}:=\Q_4e_{\varepsilon_{d}}$
        \item $M(\varphi_{\varepsilon_{d}})=1$
        \item $M(x)=-1$
        \item $M(t)=1$
        \item $\Fil^0D_{K}:=e_{\varepsilon_{d}} K$ and $\Fil^1D_{K}:=0$
    \end{itemize}
    Note that $D_{\varepsilon_{d}}$, as a filtered $(\varphi,\Gal(F/\Q_2))$-module, was a $\Q_2$-vector space while here it is a $\Q_4$-vector space. Previously, we only needed to specify the action of $x$, the generator of $\Gal(F/\Q_2)$, whereas we now need to specify how $t$ acts in order to define the whole action of $\Gal(K/\Q_2)$. Moreover, $\varphi$ was linear (the identity) and is now a semilinear map whose matrix in a properly chosen basis is $1$.

    Once we have both objects described as $\fpgm$s, we can compute their tensor
    product $D:=D_{pc}(3,0,\alpha)\otimes_{\Q_4} D_{\varepsilon_{d}}$. A basis of $D$ is given by
    $e_{1\varepsilon_{d}}:=e_1\otimes e_{\varepsilon_{d}}$ and $e_{2\varepsilon_{d}}:=e_2\otimes e_{\varepsilon_{d}}$. In this
    basis, we have $$M(\varphi\otimes\varphi_{\varepsilon_{d}})=\tbt{0}{-2}{1}{0}\quad
        M(x)=\tbt{-\zeta_3}{0}{0}{-\zeta_3^2}=\tbt{\zeta_6^{-1}}{0}{0}{\zeta_6}\quad
        M(t)=\tbt{1}{0}{0}{1}$$

    Moreover, the filtration is $$ \Fil^1(D_K)=(\alpha\otimes e_{1\varepsilon_{d}}+\sqrt[3]{2}^2\otimes e_{2\varepsilon_{d}})K $$

    To put this in the form of $D_{pc}(6,0,i,\beta)$ for some $\beta$, we
    need the eigenvalue $\zeta_6$ to correspond to the first basis vector and
    $\zeta_6^{-1}$ to the second one. Thus we can use the new basis $e'_1=e_{2\varepsilon_{d}}$
    and $e'_2=-2e_{1\varepsilon_{d}}$. In this basis
    $$M(\varphi\otimes\varphi_{\varepsilon_{d}})=\tbt{0}{-2}{1}{0}\quad M(x)=\tbt{\zeta_6}{0}{0}{\zeta_6^{-1}}\quad
        M(t)=\tbt{1}{0}{0}{1}$$

    In this basis, the filtration is $$ \Fil^1(D_K)~=~(-4\alpha^{-1}\otimes
        e'_1+\sqrt[3]{2} \otimes e'_2)K $$

    Thus $D_{pc}(3,0,\alpha)\otimes D_{\varepsilon_{d}}\simeq D_{pc}(6,0,i,-4\alpha^{-1})$
    as desired.
\end{proof}

\begin{corollary}\label{Cor: classification D6}
    Let $E/\Q_2$ be an elliptic curve with semistability defect $e=6$ and let $V_2(E)$ be its $2$-adic Tate module. Then $D_{pcris}^{*}(V_2(E))\simeq D_{pc}(6,0,i,\alpha)$ for a unique $\alpha\in\PP(\Q_2)$.
\end{corollary}
\begin{proof}
    Any elliptic curve $E/\Q_2$ with semistability defect $e=6$ is the quadratic twist of an elliptic curve $E_d$ with semistability defect $e=3$. Thus $V_2(E)=V_2(E_d)\otimes \varepsilon_{d}$ with $\varepsilon_{d}$ the quadratic character associated with the extension $F=\Q_2(\sqrt{d})$. Since $D_{pcris}^{*}$ is a tensor antiequivalence of categories, it follows that $D_{pcris}^{*}(V_2(E))\simeq D_{pcris}^{*}(V_2(E_d))\otimes D_{\varepsilon_{d}}$. By Corollary~\ref{Cor: classification for D3}, $D_{pcris}^{*}(V_2(E_d))\simeq D_{pc}(3,0,\alpha)$ for some $\alpha\in \PP(\Q_2)$  and so, by Lemma~\ref{L: Twist computation D6}, it follows that $D_{pcris}^{*}(V_2(E))\simeq D_{pc}(6,0,i,-4\alpha^{-1})$ as desired.
\end{proof}

\section{Semistability defect $e=8$}

Let $E/\Q_2$ be an elliptic curve with semistability defect $e=8$. By
Lemma~\ref{L: InTypes}, any such curve attains good reduction over $L=L_{8,i}$ for some $i$. Let $K$ be the Galois closure. Then $\Gal(K/\Q_2)\simeq SD_{16}$, the semidihedral group of order $16$. This group can be described as the semidirect
product $\Gal(K/\Q_2)\simeq I_{K/\Q_2}\rtimes \langle Frob\rangle$.
\[\Gal(K/\Q_2)=\langle x,y\rangle\rtimes\langle t \rangle=\langle
    x,y,t~|~x^4=1,~y^2=x^2,~y^{-1}xy=x^{-1},~t^2=1,~txt=y\rangle\]
where $x,y$ are generators of $I_{K/\Q_2}\simeq Q_8$ and $t$ is some Frobenius
for the extension $K/\Q_2$ that we fix for the rest of this section.

\begin{defn}
    Let $K$ be one of the fields above, let $\phi_1=\frac{1+\sqrt{5}}{2}$ be a root of $X^2-X-1$, and let $\phi_2=t(\phi_1)$. We define the $\fpgm D_{pc}(8,0,i,\alpha,\varepsilon)$ with $\varepsilon\in\{+,-\}$ by:

    \begin{itemize}
        \item $D:=\Q_4e_1\oplus\Q_4e_2$
        \item $M(\varphi)=\begin{pmatrix}
                      0 & -2 \\
                      1 & 0  \\
                  \end{pmatrix}$
        \item $M(x)=\varepsilon \begin{pmatrix}
                      \sqrt{-3} & -2\phi_2   \\
                      \phi_1    & -\sqrt{-3} \\
                  \end{pmatrix}$
        \item $M(y)=\varepsilon \begin{pmatrix}
                      -\sqrt{-3} & -2\phi_1  \\
                      \phi_2     & \sqrt{-3} \\
                  \end{pmatrix}$
        \item $M(t)=\begin{pmatrix}
                      1 & 0 \\
                      0 & 1 \\
                  \end{pmatrix}$
        \item $\Fil^1D_{K}:=(\alpha\tenso e_1 +1\tenso e_2)K$ for $\alpha\in K^{\langle x^2,t\rangle}$ such that $x(\alpha)=\frac{\alpha\sqrt{-3}-2\phi_2}{\alpha\phi_1-\sqrt{-3}}$
    \end{itemize}
\end{defn}

\begin{remark}\label{R: x,y,t}
    Note that the above description of $x,y,t$ does not determine them uniquely and that the definition of $D_{pc}(8,0,i,\alpha,\pm)$ depends strongly on the choice of $x,y,t$. For instance, if one chooses $\tilde{x},\tilde{y},t$, where $\tilde{x}=x^{-1},\tilde{y}=y^{-1}$, then the definition of $D_{pc}(8,0,i,\alpha,+)$ in terms of $\tilde{x},\tilde{y},t$ would have $$M(\tilde{x})=\begin{pmatrix}
                      \sqrt{-3} & -2\phi_2   \\
                      \phi_1    & -\sqrt{-3} \\
                  \end{pmatrix}\quad\text{and}\quad M(\tilde{y})=\begin{pmatrix}
                      -\sqrt{-3} & -2\phi_1  \\
                      \phi_2     & \sqrt{-3} \\
                  \end{pmatrix}$$ which implies $$M(x)=\begin{pmatrix}
                      -\sqrt{-3} & 2\phi_2   \\
                      -\phi_1    & \sqrt{-3} \\
                  \end{pmatrix}\quad\text{and}\quad M(y)=\begin{pmatrix}
                      \sqrt{-3} & 2\phi_1    \\
                      -\phi_2   & -\sqrt{-3} \\
                  \end{pmatrix}.$$ We see that the $\fpgm$ described as $D_{pc}(8,0,i,\alpha,+)$ in terms of $\tilde{x},\tilde{y},t$ is the module described as $D_{pc}(8,0,i,\alpha,-)$ in terms of $x,y,t$; the two modules are nonisomorphic. For this reason, throughout this section we keep $x,y,t$ fixed.
 
\end{remark}

\begin{proposition}\label{P: classification SD16}
       Let $L=L_{8,i}$ be an extension in Table~\ref{Ta: InTypes}. Let $D$ be a $\fpgmt$ satisfying ${\bf (A.2)}$ for $K/L$. Then $D\simeq D_{pc}(8,0,i,\alpha,\varepsilon)$ for some $\varepsilon\in\{+,-\}$ as an $\upgmt$. If, moreover, $D$ has Hodge--Tate weights $(0,1)$, then $D\simeq D_{pc}(8,0,i,\alpha,+)$ or $D\simeq D_{pc}(8,0,i,\alpha,-)$ for a unique $\alpha$.
\end{proposition}

\begin{proof}
    We work in the basis provided by~\ref{L: adapted basis}. Since $\zeta_4+\zeta_4^{-1}=0$, Lemma~\ref{L: eigenvalues} implies that the minimal polynomial of each of $x$ and $y$ acting on $D$ is $X^2+1$. Thus, in this basis,

    $$M(x)=\begin{pmatrix}
            \lambda & \mu'     \\
            \mu     & -\lambda \\
        \end{pmatrix}\quad M(y)=\begin{pmatrix}
            \gamma & \delta' \\
            \delta & -\gamma \\
        \end{pmatrix}\quad M(t)=\begin{pmatrix}
            1 & 0 \\
            0 & 1 \\
        \end{pmatrix}$$

    In particular, $M(x)^{-1}=-M(x)$ and $M(y)^{-1}=-M(y)$. Moreover, by Lemma~\ref{L: a_2=0}, we have $a_2=0$.

    From the relation $txt=y$, or equivalently $tx=yt$, we obtain $M(x)=t(M(y))$, or equivalently: $$t(\lambda)=\gamma\quad t(\mu)=\delta\quad t(\mu')=\delta'$$ From $\varphi x=x\varphi$ and $\varphi y=y\varphi$ we obtain
   respectively,
$$  t(\lambda)=-\lambda,\qquad
    t(\mu)=-\frac{\mu'}{2},\qquad
    t(\mu')=-2\mu,$$
and
$$
    t(\gamma)=-\gamma,\qquad
    t(\delta)=-\frac{\delta'}{2},\qquad
    t(\delta')=-2\delta.
$$

   Since $t(\lambda)=-\lambda$, we have $\lambda=u_1\sqrt{-3}$ with $u_1\in\Q_2$. The only equation imposed by the relation $xy=yx^{-1}$ that does not follow from the previous equations is $$ 3u_1^2+\mu'\delta=-3u_1^2-\mu\delta'.$$ Since $\delta=t(\mu)=\frac{-\mu'}{2}$ and
    $\delta'=t(\mu')=-2\mu$, the above equations are equivalent to $$
        6u_1^2=\frac{(\mu')^2}{2}+2\mu^2 $$ Then the condition $\det(M(x))=1$ reads
    $3u_1^2-\mu\mu'=1$. Substituting this into the previous equation gives $$
        (\mu-\frac{\mu'}{2})^2=1 \iff -\frac{\mu'}{2}=t(\mu)=\pm 1-\mu.$$ We will show later that the two
    possible signs give nonisomorphic solutions. Since the reasoning is analogous in both cases, we work out only
    the case with the $+$ sign. In this case, $\mu=\frac{1}{2}+u_2\sqrt{5}$ with $u_2\in\Q_2$, and
    so $\mu'=-1+2u_2\sqrt{5}$. Substituting these expressions back into $\det(M(x))=1$, we get the equation $$ u_2=\pm\sqrt{\frac{1-6u_1^2}{-20}}.$$ Thus the matrices of $x$
    and $y$ for $D_{pcris}^{*}(V_2(E))$ are of the form

    $$M(x)=\begin{pmatrix}
            u_1\sqrt{-3}            & -1+2u_2\sqrt{5} \\
            \frac{1}{2}+u_2\sqrt{5} & -u_1\sqrt{-3}   \\
        \end{pmatrix}\quad M(y)=\begin{pmatrix}
            -u_1\sqrt{-3}           & -1-2u_2\sqrt{5} \\
            \frac{1}{2}-u_2\sqrt{5} & u_1\sqrt{-3}    \\
        \end{pmatrix}$$

    The matrix $$B=\begin{pmatrix}
            3(u_1+1)                     & 2\sqrt{-15}(u_2-\frac{1}{2}) \\
            -\sqrt{-15}(u_2-\frac{1}{2}) & 3(u_1+1)                     \\
        \end{pmatrix}~\text{if}~(u_1,u_2)\neq(-1,\frac{1}{2})\quad\text{and}\quad B=\tbt{-\sqrt{-15}}{-6}{3}{-\sqrt{-15}}~\text{otherwise}$$ defines a change of basis that sends the above matrices to the ones in the definition of $D_{pc}(8,0,\alpha,+)$. The case $t(\mu)=-1-\mu$ follows analogously and the computations can be found in~\cite{MAGMAFiles}. This proves the unfiltered isomorphism.

    By Lemma~\ref{L: invariant subspaces}, any $\Gal(K/\Q_2)$-stable line defines an admissible filtration. The relation $y=txt$ implies that a line stable under $x$ and $t$ is automatically stable under $y$ and thus under $\Gal(K/\Q_2)$, so we only need to impose stability under $t$ and $x$. The line generated by $e_1$ is not stable under $\Gal(K/\Q_2)$, so we can assume that $\Fil^{1}D_K$ is generated by $\alpha\otimes e_1+1\otimes e_2$. Such a line is stable under $t$ if and only if $t(\alpha)=\alpha$; thus $\alpha\in
        K^{\langle t \rangle}=L$. It follows that such a line is stable under $\Gal(K/\Q_2)$ if and only if $\alpha\in L$ and $\alpha=\frac{x(\alpha)\sqrt{-3}-2\phi_2}{x(\alpha)\phi_1-\sqrt{-3}}$ or equivalently $x(\alpha)=\frac{\alpha\sqrt{-3}-2\phi_2}{\alpha\phi_1-\sqrt{-3}}.$ From this equation, it follows that $x^2(\alpha)=\alpha$, and so $y^2(\alpha)=\alpha$; thus $\alpha$ is fixed by $x^2=y^2$, so $\alpha\in K^{\langle t,x^2 \rangle}$.

    Each such $\alpha$ gives rise to a nonisomorphic $\fpgm$ by Lemma~\ref{L: unfiltered automorphism}.

    For $t(\mu)=-1-\mu$, the same reasoning gives $D_{pc}(8,0,i,\alpha,-)$, which is not isomorphic to $D_{pc}(8,0,i,\alpha,+)$. Indeed, an isomorphism $f$ between the two would satisfy $tf=ft$ so it must be defined over $\Q_2$ and then $\varphi\circ f=f\circ \varphi$ implies $M(f)=c_0Id+c_1M(\varphi)$, but there are no nonzero values of $c_0,c_1$ for which $M(f)M(x)^+=M(x)^{-}M(f)$.

    See~\cite{MAGMAFiles} for all the explicit computations of this proof.
\end{proof}

\begin{corollary}\label{Cor: classification for SD16}
    Let $E/\Q_2$ have inertia type associated with $L_{8,i}$. Then $D_{pcris}^{*}(V_2(E))\simeq D_{pc}(8,0,i,\alpha,+)$ or $D_{pcris}^{*}(V_2(E))\simeq D_{pc}(8,0,i,\alpha,-)$ for a unique $\alpha$.
\end{corollary}
\begin{proof}
    This follows directly from Lemma~\ref{L: curves satisfy A} and Proposition~\ref{P: classification SD16}.  
\end{proof}

\begin{lemma}\label{L: Twist for e=8}
    $D_{pc}(8,0,i,\alpha,+)\otimes D_{\eta}=D_{pc}(8,0,i,\frac{\sqrt{-15}\alpha+6}{3\alpha-\sqrt{-15}},-)$
\end{lemma}
\begin{proof}
    Let $D:=D_{pc}(8,0,i,\alpha,+)\otimes D_{\eta}$ and denote by $M^+,M^-$ and $M$ the matrices describing the operators of $D_{pc}(8,0,i,\alpha,+),D_{pc}(8,0,i,\alpha,-)$ and $D$, respectively.

    The $\fpgmt$ corresponding to $D_{\eta}$ is $D_{\eta}=\Q_4e_{\eta}$, $M(\varphi)=-\Id$, $M(x)=M(y)=M(t)=\Id$ (recall that $t$ acts semilinearly) with filtration $\Fil^{0}D_{\eta,K}=D_{\eta,K}$ and $\Fil^{1}D_{\eta,K}=0$. Thus, we have $$M(\varphi)=-M^{+}(\varphi),~M(x)=M^{+}(x), M(y)=M^{+}(y)\quad\text{and}\quad \Fil^1D_{K}=(\alpha\tenso e_1 +1\tenso e_2)K.$$
    The change of basis given by the matrix $B=\tbt{\sqrt{-15}}{6}{3}{-\sqrt{-15}}\in M_2(\Q_2)$ satisfies $$M(\varphi)B=B M^{-}(\varphi),~M(x)B=BM^{-}(x),~M(y)B=BM^{-}(y),$$ and it sends the line $(\alpha\tenso e_1 +1\tenso e_2)K$ to the line $(\beta\tenso e_1 +1\tenso e_2)K$, where $\beta=\frac{\sqrt{-15}\alpha+6}{3\alpha-\sqrt{-15}}$. It can be checked that if $x(\alpha)=\frac{\alpha\sqrt{-3}-2\phi_2}{\alpha\phi_1-\sqrt{-3}}$, then $x(\beta)=\frac{\beta\sqrt{-3}-2\phi_2}{\beta\phi_1-\sqrt{-3}}$ and so the result follows.

\end{proof}

\section{Semistability defect $e=24$}

Let $E/\Q_2$ be an elliptic curve with semistability defect $e=24$. By
Lemma~\ref{L: InTypes}, any such curve attains good reduction over $L=L_{24,i}$ for some $i$. Let $K$ be the Galois closure. Then $\Gal(K/\Q_2)\simeq \GL_2(\F_3)$. This group can be described as
a semidirect product
\[\Gal(K/\Q_2)\simeq I_{K/\Q_2}\rtimes\langle Frob \rangle =\langle x,y\rangle\rtimes\langle t \rangle=\langle
    x,y,t~|~x^3=y^3=(xy)^2,~t^2=1,~txt=y\rangle\]

    Although $x,y,t$ are not uniquely determined by the above description, as in the case $e=8$, we fix a choice of $x,y,t$ for the rest of this section; see Remark~\ref{R: x,y,t}.
\begin{defn}
    Let $K$ be one of the fields described above and let $\zeta_6=\frac{1+\sqrt{-3}}{2}$ be a primitive sixth root of unity. We define the $\fpgm D_{pc}(24,0,i,\alpha,\pm)$ by:

    \begin{itemize}
        \item $D:=\Q_4e_1\oplus\Q_4e_2$
        \item $M(\varphi)=\begin{pmatrix}
                      0 & -2 \\
                      1 & 0  \\
                  \end{pmatrix}$
        \item $M(x)=\begin{pmatrix}
                      \phi_1    & -2\zeta_6^{-1} \\
                      \zeta_{6} & \phi_2         \\
                  \end{pmatrix}$  for $D_{pc}(24,0,i,\alpha,+)$ $\quad$ and $\quad$ $M(x)=\begin{pmatrix}
                      \phi_1     & 2\zeta_6^{-1} \\
                      -\zeta_{6} & \phi_2        \\
                  \end{pmatrix}$ for $D_{pc}(24,0,i,\alpha,-)$
        \item $M(y)=\begin{pmatrix}
                      \phi_2         & -2\zeta_6 \\
                      \zeta_{6}^{-1} & \phi_1    \\
                  \end{pmatrix}$ for $D_{pc}(24,0,i,\alpha,+)$ $\quad$ and $\quad$ $M(y)=\begin{pmatrix}
                      \phi_2          & 2\zeta_6 \\
                      -\zeta_{6}^{-1} & \phi_1   \\
                  \end{pmatrix}$ for $D_{pc}(24,0,i,\alpha,-)$
        \item $M(t)=\begin{pmatrix}
                      1 & 0 \\
                      0 & 1 \\
                  \end{pmatrix}$
        \item $\Fil^1D_{K}:=(\alpha\tenso e_1 +1\tenso e_2)K$ for $\alpha\in K^{\langle x^3,t\rangle}$ such that $$ x(\alpha)=\frac{-\phi_2\alpha-2\zeta_6^{-1}}{\zeta_6\alpha-\phi_1}\quad\text{for}~D_{pc}(24,0,i,\alpha,+)\quad\quad x(\alpha)=\frac{\phi_2\alpha-2\zeta_6^{-1}}{\zeta_6\alpha+\phi_1}\quad\text{for}~D_{pc}(24,0,i,\alpha,-)$$
    \end{itemize}
\end{defn}

\begin{proposition}\label{P: classification GL(2,3)}
       Let $L=L_{24,i}$ be an extension in Table~\ref{Ta: InTypes}. Let $D$ be a $\fpgmt$ satisfying ${\bf (A.2)}$ for $K/L$. Then $D\simeq D_{pc}(24,0,i,\alpha,\varepsilon)$ for some $\varepsilon\in\{+,-\}$ as an $\upgmt$. If, moreover, $D$ has Hodge--Tate weights $(0,1)$, then $D\simeq D_{pc}(24,0,i,\alpha,+)$ or $D\simeq D_{pc}(24,0,i,\alpha,-)$ for a unique $\alpha$.
\end{proposition}

\begin{proof}
    We work in the basis provided by~\ref{L: adapted basis}. Since $\zeta_6+\zeta_6^{-1}=1$, Lemma~\ref{L: eigenvalues} implies that the minimal polynomial of each of $x$ and $y$ acting on $D$ is $X^2-X+1$. Thus, in this basis,

   $$M(x)=\begin{pmatrix}
            \lambda & \mu'      \\
            \mu     & 1-\lambda \\
        \end{pmatrix}\quad M(y)=\begin{pmatrix}
            \gamma & \delta'  \\
            \delta & 1-\gamma \\
        \end{pmatrix}\quad M(t)=\begin{pmatrix}
            1 & 0 \\
            0 & 1 \\
        \end{pmatrix}$$

    In particular, $M(x)^{3}=-\Id$ and $M(y)^{3}=-\Id$. Moreover, by Lemma~\ref{L: a_2=0}, we have $a_2=0$.

    From the condition $tx=yt$, we deduce that $t(M(x))=M(y)$; that is, $$
        t(\lambda)=\gamma,\quad t(\mu)=\delta,\quad t(\mu')=\delta'. $$ 
        From $\varphi\circ x=x\circ \varphi$ and $\varphi\circ y=y\circ \varphi$ we obtain the equations $$-2t(\mu)=\mu',\quad -2+2t(\lambda)=-2\lambda,\quad
        t(\lambda)=1-\lambda,\quad t(\mu')=-2\mu$$ $$-2t(\delta)=\delta',\quad
        -2+2t(\gamma)=-2\gamma,\quad t(\gamma)=1-\gamma,\quad t(\delta')=-2\delta$$
        These equations directly imply $\lambda=\frac{1}{2}+u_1\sqrt{5}$  with $u_1\in \Q_2$ and, combined with those arising from $tx=yt$, yield $$ \gamma=1-\lambda,\quad
        \delta=\frac{-\mu'}{2},\quad \delta'=-2\mu. $$

    The condition $xyxy=x^3=-Id$ implies $xy=-(xy)^{-1}$. In particular, this imposes the condition $$ -\lambda^2+\lambda-\frac{\mu'^2}{4}-\mu^2=0 $$ which, using $\det(x)=1$, can be rewritten as $$ (\mu-\frac{\mu'}{2})^2=1  \iff -\frac{\mu'}{2}=t(\mu)=\pm 1-\mu$$ As in the case $e=8$, we give the details only for the $+$ case, the $-$ case
    being completely analogous.
    
    In the $+$ case, $\mu=\frac{1}{2}+u_2\sqrt{-3}$ with $u_2\in\Q_2$, and
    $\mu'=-1+2u_2\sqrt{-3}$. Moreover, substituting back into $\det(x)=1$, we
    obtain $u_2=\pm\frac{1}{2}\sqrt{\frac{1+20u_1^2}{6}}$. Thus the matrices of $x$ and $y$ for $D_{pcris}^{*}(V_2(E))$ are of the form $$M(x)=\begin{pmatrix}
            \frac{1}{2}+u_1\sqrt{5}  & -1+2u_2\sqrt{-3}        \\
            \frac{1}{2}+u_2\sqrt{-3} & \frac{1}{2}-u_1\sqrt{5} \\
        \end{pmatrix}\quad M(y)=\begin{pmatrix}
            \frac{1}{2}-u_1\sqrt{5}  & -1-2u_2\sqrt{-3}        \\
            \frac{1}{2}-u_2\sqrt{-3} & \frac{1}{2}+u_1\sqrt{5} \\
        \end{pmatrix}$$

    The matrix
    $$B=\tbt{5(u_1+\frac{1}{2})}{-2\sqrt{-15}(u_2-\frac{1}{2})}{\sqrt{-15}(u_2-\frac{1}{2})}{5(u_1+\frac{1}{2})}~\text{if}~(u_1,u_2)\neq(-\frac{1}{2},\frac{1}{2})\quad\text{and}\quad B=\tbt{2\sqrt{-15}}{-10}{5}{2\sqrt{-15}}~\text{otherwise}$$ defines a change of basis that sends these matrices to the ones in the
    definition of $D_{pc}(24,0,i,\alpha,+)$.

 Moreover, the two possible choices of sign
    for $u_2$ produce isomorphic objects. See~\cite{MAGMAFiles} for the explicit
    computations. By Lemma~\ref{L: invariant subspaces}, any
    $\Gal(K/\Q_2)$-stable $K$-line in $D_K$ leads to an admissible $\fpgm$.

     Since $y=txt$, it is enough to look for lines stable under $t$ and $x$. The line
    $e_1K$ is not $\Gal(K/\Q_2)$-stable, so we can assume that
    $\Fil^{1}D_{K}=(\alpha\otimes e_1+1\otimes e_2)K$. Such a line is stable under $t$
    if and only if $t(\alpha)=\alpha$, and it is stable under $x$ if and 
    only if $$\alpha=\frac{\phi_1x(\alpha)-2\zeta_6^{-1}}{\zeta_6x(\alpha)+\phi_2}$$ which is equivalent to $$ x(\alpha)=\frac{-\phi_2\alpha-2\zeta_6^{-1}}{\zeta_6\alpha-\phi_1}$$

    Applying $x$ repeatedly, we see that $x^3(\alpha)=\alpha$; thus $\alpha$ is fixed by $x^3=y^3$ and by $t$, so $\alpha\in
        K^{\langle x^3,t \rangle}$ as desired.

    Each such $\alpha$ gives rise to a nonisomorphic $\fpgm$ by Lemma~\ref{L: unfiltered automorphism}.
    
    In the case $t(\mu)=-1-\mu$, the same arguments lead to $D_{pc}(24,0,i,\alpha,-)$. As in the case $e=8$, one can check that $D_{pc}(24,0,i,\alpha,+)$ and $D_{pc}(24,0,i,\alpha,-)$ are nonisomorphic.
\end{proof}

\begin{corollary}\label{Cor: classification for GL(2,3)}
    Let $E/\Q_2$ have inertia type associated with $L_{24,i}$. Then $D_{pcris}^{*}(V_2(E))\simeq D_{pc}(24,0,i,\alpha,+)$ or $D_{pcris}^{*}(V_2(E))\simeq D_{pc}(24,0,i,\alpha,-)$ for a unique $\alpha$.
\end{corollary}
\begin{proof}
    This follows directly from Lemma~\ref{L: curves satisfy A} and Proposition~\ref{P: classification GL(2,3)}.  
\end{proof}

\begin{lemma}\label{L: Twist for e=24}
    $D_{pc}(24,0,i,\alpha,+)\otimes D_{\eta}=D_{pc}(24,0,i,-\alpha,-)$
\end{lemma}
\begin{proof}
    The proof is analogous to that of Lemma~\ref{L: Twist for e=8}, with $B=\tbt{1}{0}{0}{-1}$, hence the difference in the behavior of the filtration.
\end{proof}
\section{Construction of curves}\label{S: explicit realization}

The combination of the results in~\ref{P: e=1ordinary}, \ref{P: e=1supersingular}, \ref{P: e=2}, \ref{Cor: classification for D3}, \ref{Cor: classification for D4}, \ref{Cor: classification for C4}, \ref{Cor: classification D6}, \ref{Cor: classification for SD16} and \ref{Cor: classification for GL(2,3)} shows that any elliptic curve $E/\Q_2$ with potentially good reduction $D_{pcris}^{*}(V_2(E))$ must be among the $\fpgmt$s in Table~\ref{Ta: Main} and so it proves one implication of Theorem~\ref{T: Main Theorem}. This section is devoted to the proof of the converse, that is, for any $D$ in Table~\ref{Ta: Main} there exists an elliptic curve $E/\Q_2$ such that $D_{pcris}^{*}(V_2(E))\simeq D$. To do so we will follow the approach of \cite{Volkov2005} and \cite{bosco20233adicrepresentationsarisingelliptic} using Galois pairs and Volkov descend criterion~\cite[Theorem 1.2]{Volkov2005}.

\subsection{Dieudonn\'e modules and Galois pairs}
In this subsection we briefly recall the Dieudonn\'e theory needed to describe Galois pairs. See~\cite{Volkov2005} for a detailed exposition.

 Let $k$ be a perfect field of characteristic $p$, $W(k)$ the ring of Witt vectors, $\sigma$ its Frobenius, and $K_0=\Frac(W(k))$. A Dieudonn\'e module $D$ is a finite free $W(k)$-module equipped with a $\sigma$-semilinear map $\varphi$ and a $\sigma^{-1}$-semilinear map $V$ such that $\varphi V=V\varphi=p$. A morphism of Dieudonn\'e modules is a $W(k)$-linear map commuting with the corresponding $\varphi$ and $V$. There is an antiequivalence of categories $\ma D$ from the category of $p$-divisible groups $\cl G$ over $k$ to the category of Dieudonn\'e modules. Moreover, after inverting $p$, the functor $D(\cl G):=\ma D(\cl G)\otimes_{W(k)}K_0$ defines an antiequivalence of categories between $p$-divisible groups up to quasi-isogeny and rational Dieudonn\'e modules, that is, $K_0$-vector spaces $D$ with a bijective $\sigma$-semilinear map $\varphi$ such that $D$ admits a finite free lattice $M$ satisfying $pM\subset \varphi(M)\subset M$. Thus, for $k=\F_{p^s}$, if we equip $D(\cl G)$ with an appropriate $\Gal(K/\Q_p)$-action, it becomes an unfiltered $(\varphi,\Gal(K/\Q_p))$-module.

 We use the following properties:
 \begin{itemize}
    
    \item The functor $\ma D$ is natural with respect to base change, that is, there exist a canonical isomorphism $\ma D(\cl G\otimes_{k}k')\simeq D(\cl G)\otimes_{W(k)}\otimes W(k')$. In particular, the functor preserves twists, for $g\in\Gal(K_0/\Q_p)$ and $\overline{g}$ its class in $\Gal(k/\F_p)$ there is a canonical isomorphism $\iota_g:\ma D(\cl G)^g\xrightarrow{\sim}\ma D(\cl G^{\bar{g}})$, see \cite[\S 4.3]{Berthelot1980}.
    \item $\operatorname{rank}_{W(k)}\ma D(H)=\operatorname{height}(H)$; see \cite[\S 4.3.1]{Berthelot1980}.
    \item The crystalline Frobenius $\varphi:D(\cl G)\rightarrow D(\cl G)$ can be defined as $D(F/k)\circ \iota_{\sigma}\circ m_{\sigma}$, where $F:\cl G\rightarrow \cl G^{\sigma}$ is the $p$-th power relative Frobenius and $m_{\sigma}$ is as defined in the notation subsection; see \cite[\S 4.3.2]{Berthelot1980}.  
 \end{itemize}

From now on, given an elliptic curve $E/\F_{p^s}$, we use $D(E)$ to denote the rational Dieudonn\'e module of its $p$-divisible group, namely $\ma D(E[p^{\infty}])[1/p]$. The above properties imply that $D(E)$ is a $2$-dimensional $K_0$-vector space with a $\sigma$-semilinear map $\varphi$. It is possible to equip $D(E)$ with the structure of an unfiltered $(\varphi,\Gal(K/\Q_p))$-module. This is precisely the purpose of the Galois pairs defined by Volkov in~\cite{Volkov2005}, which we now introduce.

Let $K/\Q_p$ be a finite extension with residue field $\F_{p^s}$ and let $\sigma\in \Gal(\F_{p^s}/\F_p)$ be the Frobenius. Let $E/\F_{p^s}$ be an elliptic curve. The group $\operatorname{Ist}_{\F_{p^s}}(E)$ consists of isomorphisms from $E$ to its $\Gal(\F_{p^s}/\F_p)$-twists. That is, $\operatorname{Ist}_{\F_{p^s}}(E)=\{\psi_{g}:=(\psi,g)\st \psi:E\simeq E^{g}\}$ and the composition law is defined by $\psi_g*\psi_h=(\psi_g)^h\circ\psi_h$. The map $\operatorname{pr}(\psi_g)=g$ defines a group homomorphism $\operatorname{pr}:\operatorname{Ist}_{\F_{p^s}}(E)\rightarrow \Gal(\F_{p^s}/\F_p)$. If $E$ is the base change of a curve $E_0/\F_p$, that is, $E=E_0\times_{\F_p}\F_{p^s}$, then $E^{\sigma}=E_0\times_{\F_p,\sigma}\F_{p^s}$, so $f=Id\times\sigma$ gives rise to an element $f_{\sigma}=(f,\sigma)\in \operatorname{Ist}_{\F_{p^s}}(E)$ and the map $\sigma\mapsto f_{\sigma}$ is a section of $\operatorname{pr}$; thus $\operatorname{Ist}(E)\simeq Aut_{\F_{p}^s}(E)\rtimes \Gal(\F_{p^s}/\F_p)$. Moreover, in this situation, the relative $p$-th power Frobenius $F:E\rightarrow E^{\sigma}$ splits as $F=f_{\sigma}\circ (F_0\otimes k)$, where $F_0:E_{0}\rightarrow E_{0}$ is the absolute Frobenius of $E_0$.

\begin{defn}\label{Def: Galois pairs}
    Let $K/\Q_p$ be a finite Galois extension. A Galois pair $(E_0,\Gamma,\nu)$ for $K/\Q_p$ consists of an elliptic curve $E_0/\F_p$, a subgroup $\Gamma$ of $Aut_{\F_{p^s}}(E)$ with $E=E_0\times_{\F_p}\F_{p^s}$ and an antihomomorphism $\nu:\Gal(K/\Q_p)\rightarrow Aut_{\F_{p^s}}(E)\rtimes\Gal(\F_{p^s}/\F_p)$ such that:
 \begin{itemize}
        \item $\operatorname{pr}\circ \nu(g)= \bar{g}$, where $pr$ is the projection onto the second factor and $\bar{g}$ is the class of $g$ in $\Gal(\F_{p^s}/\F_p)$.

        \item $\Image(\nu)=\Gamma\rtimes\Gal(\F_{p^s}/\F_p)$.
    \end{itemize}
 An isogeny between the Galois pairs $(E_0,\Gamma,\nu)$ and $(E'_0,\Gamma',\nu')$ is an isogeny $f:E_0\longrightarrow E'_0$ defined over $\F_p$ such that $\nu'(g)\circ f=f^{\bar{g}}\circ \nu(g)$ for all $g\in\Gal(K/\Q_p)$.

 A Galois pair is called minimal if $\nu$ is injective and, for the given $E_0$, $\F_{p^s}$ is the smallest field for which $\Gamma<Aut_{\F_{p^s}}(E)$.
\end{defn}

A Galois pair $(E_0,\Gamma,\nu)$ produces an unfiltered $(\varphi,\Gal(K/\Q_p))$-module as follows. Each $(\psi,g):E\rightarrow E^{\bar{g}}$ induces, via the Dieudonn\'e functor, a $K_0$-linear map $D(\psi_g):D(E^{\overline{g}})\rightarrow D(E)$. Since the action of $K_0$ on $D(E)^{g}$ is twisted by $g$, this $K_0$-linear map induces a $g$-semilinear map $c_{\psi}=D(\psi)\circ\iota_g\circ m_g$ from $D(E)$. Composing the antihomomorphism $\nu:\Gal(K/\Q_p)\rightarrow \operatorname{Ist}_{\F_{p^s}}(E)$ with the antiequivalence of categories given by the rational Dieudonn\'e functor gives a map $g\rightarrow c_{\nu(g)}$ that defines the semilinear action of $\Gal(K/\Q_p)$ on $D(E)$. Moreover, base change gives $D(E)\simeq D(E_0)\otimes_{\Q_p}K_0$. For an isomorphism $\psi:E\longrightarrow E^{\bar g}$, naturality of relative Frobenius gives $F_{E^{\bar g}/k}\circ\psi=\psi^\sigma\circ F_{E/k},$ where the twisted targets are canonically identified because $\bar g$ and $\sigma$ commute. Applying Dieudonn\'e theory and its base-change isomorphisms to this identity gives
$c_\psi\varphi=\varphi c_\psi$. Hence the resulting semilinear Galois action commutes with $\varphi$, making $D(E)$ an unfiltered $(\varphi,\Gal(K/\Q_p))$-module. Under this construction, isogenous Galois pairs lead to isomorphic unfiltered $(\varphi,\Gal(K/\Q_p))$-modules.

\subsection{Construction of Galois pairs}

Up to $\F_2$-isomorphism, there are $5$ elliptic curves $E/\F_2$, one for each of the
possible traces of Frobenius. We list them here together with some properties that will be useful later.

\begin{table}[h]
    \centering
    \begin{tabular}{lccc}
        \hline
        Curve equation
                               & Trace of Frob.
                               & $\operatorname{Aut}_{\mathbb{F}_2}(E)$
                               & $\operatorname{Aut}_{\mathbb{F}_4}(E)$                              \\
        \hline
        $E_1:\ y^2+y=x^3+x$    & $-2$                                   & $C_4$ & $C_4$              \\
        $E_2:\ y^2+xy=x^3+1$   & $-1$                                   & $C_2$ & $C_2$              \\
        $E_3:\ y^2+y=x^3$      & $0$                                    & $C_2$ & $\mathrm{SL}(2,3)$ \\
        $E_4:\ y^2+xy+y=x^3+1$ & $1$                                    & $C_2$ & $C_2$              \\
        $E_5:\ y^2+y=x^3+x+1$  & $2$                                    & $C_4$ & $C_4$              \\
        \hline
    \end{tabular}
    \caption{Relevant data for elliptic curves $E/\F_2$.}\label{Table: Ell curves over F2}
\end{table}

\begin{lemma}\label{L: Automorphism Group E3}
    Let $\omega\in \F_4$ be a root of $X^2+X+1$, let $E_3: y^2+y=x^3$, and let $$\psi_1(x,y)=(\omega x+\omega,y+x+\omega),\quad \psi_2(x,y)=(\omega^2x+\omega^2,y+x+\omega^2),\quad z(x,y)=(x,y+1)$$ be automorphisms of $E_3$. Then $\Aut_{\F_4}(E_3)=\langle \psi_1,\psi_2~\vert~ \psi_1^3=\psi_2^3=(\psi_1\psi_2)^2 \rangle$ and $\Aut_{\F_4}(E_3)\rtimes\Gal(\F_4/\F_2)=\langle \psi_1,\psi_2 \rangle\rtimes \langle f\rangle=\langle \psi_1,\psi_2,f~\vert~ \psi_1^3=\psi_2^3=(\psi_1\psi_2)^2,~f^2=1,~ f\psi_1 f^{-1}=\psi_2 \rangle\simeq \GL_2(\F_3)$.
\end{lemma}

\begin{proof}
    This follows by direct computation; see~\cite{MAGMAFiles}.

\end{proof}

\begin{remark}\label{R: reduce to 3,4,8,24}
    There is no need to construct elliptic curves separately for the $\fpgmt$s with $e=2$ and $e=6$, since curves with these representations are obtained by twisting curves in the cases $e=1$ and $e=3$. Curves with $e=1$ are covered by Theorem~\ref{T: realization e=1} hence, we only need to show the existence of curves for the cases $e\in\{3,4,8,24\}$. Moreover, by Lemmas~\ref{L: Twist for e=8} and~\ref{L: Twist for e=24}, in the cases $e=8$ and $e=24$ it suffices to construct curves $E$ for the $+$ or $-$ cases, since the others follow by twisting.
  
\end{remark}

\begin{lemma}\label{L: Galois pairs exist}
    For $D\in\{D_{pc}(3,0,\alpha),D_{pc}(4,0,i,\alpha),D_{pc}(4,a_2,i,\varepsilon)\}$, there exists a Galois pair $(E_0,\Gamma,\nu)$ for which the unfiltered $(\varphi,\Gal(K/\Q_2))$-module $D(E)$ is isomorphic to the underlying unfiltered module of $D$. Moreover, there exists a Galois pair with $D(E)$ isomorphic to either $D_{pc}(8,0,i,\alpha,+)$ or $D_{pc}(8,0,i,\alpha,-)$ and a Galois pair with $D(E)$ isomorphic to either $D_{pc}(24,0,i,\alpha,+)$ or to $D_{pc}(24,0,i,\alpha,-)$.
\end{lemma}
\begin{proof}
    Note that, by Lemma~\ref{L: pot supersingular} $e\geq 3$ implies that $a_2\equiv 0\pmod 2$, so we only need to consider the supersingular cases.

    We first construct the Galois pairs for the $\fpgmt$s with $a_2=0$. For all of them, let $E_0:=E_3$, $E=E_0\otimes_{\F_2}\F_4$ and define $\Gamma$
    and $\nu$ as specified in the following table.

    \begin{table}[h!]
    \centering
    \begin{tabular}{c|c|c|c|c|c}
        $e$ & $\operatorname{Gal}(K/\mathbb{Q}_2)$
            & $\Gamma$ & $\nu(x)$ & $\nu(y)$ & $\nu(t)$ \\
        \hline
        $3$  & $S_3$ & $C_3$
             & $\psi_1\psi_2^2$
             & --- & $f$ \\
        $4$  & $D_4$ & $C_4$
             & $\psi_2\psi_1^2\psi_2$
             & --- & $f$ \\
        $8$  & $\mathrm{SD}_{16}$ & $Q_8$
             & $z\psi_1\psi_2$
             & $z\psi_2\psi_1$ & $f$ \\
        $24$ & $\mathrm{GL}_2(\mathbb{F}_3)$
             & $\mathrm{SL}_2(\mathbb{F}_3)$
             & $z\psi_1^2$
             & $z\psi_2^2$ & $f$
    \end{tabular}
    \caption{Galois pairs for $a_2=0$.}
    \label{Ta:GaloisPairs}
\end{table}
    All these antihomomorphisms are injective, see~\cite{MAGMAFiles}. The minimality of $\F_4$ for these $\Gamma$ follows from $\Aut_{\F_2}(E_3)\simeq C_2$; thus all of them are minimal Galois pairs. We need to show that the unfiltered $(\varphi,\Gal(K/\Q_2))$-module that they produce is the desired one. Since $E=E_0\otimes_{\F_2}\F_4$, compatibility with base change implies that $D(E)=D(E_0)\otimes_{\Q_2}K_0$, $D(E)^{\Gal(K/L)}=D(E_0)$ and $\varphi=\varphi_0\otimes\sigma$. By definition, $\varphi_0:D(E_0)\xrightarrow{}D(E_0)$ is $\varphi_0=D(F_0)$, where $F_0:E_0\rightarrow E_0$ is the Frobenius endomorphism. Since $F_0^2-a_2F_0+[2]=0$, applying the Dieudonn\'e functor to this expression gives $\varphi_0^2-a_2\varphi_0+2=0$. Since $\nu$ is injective $I_{K/\Q_2}$ acts faithfully on $D(E)$. We claim that $\det D(\nu(g))=1$ for all $g\in I_{K/\Q_2}$, thus $D(E)$ will satisfy property ${\bf (A)}$ and it will follow from Propositions~\ref{P: classification D3}, \ref{P: classification D4}, \ref{P: classification SD16}, and \ref{P: classification GL(2,3)} that the unfiltered $(\varphi,\Gal(K/\Q_2))$-modules produced by these Galois pairs are isomorphic to the underlying unfiltered modules of $D_{pc}(3,0,\alpha)$, $D_{pc}(4,0,i,\alpha)$, $D_{pc}(8,0,i,\alpha,\pm)$, and $D_{pc}(24,0,i,\alpha,\pm)$. In the cases $e=8$ and $e=24$, we do not know which sign our Galois pair realizes. This does not affect the argument, since the existence of curves for one sign implies, by twisting, the existence of curves for the other.
    
    To prove the claim, we will show that $\det D(\psi_1)=1$ and $\det D(\psi_2)=1$. Since, for all $\nu$, we have $\nu(I_{K/\Q_2})\subset \langle\psi_1,\psi_2 \rangle$, the claim will follow. Since $D(\psi_i)\varphi=\varphi D(\psi_i)$, taking determinants and using the fact that $\varphi$ is $\sigma$-semilinear shows that $\det D(\psi_i)\in \Q_2$. Since $\psi_i^3=z=[-1]$ then $D(\psi_i)^3=-Id$, taking determinants gives $\det D(\psi_i)^3=1$. Since we have shown that $\det D(\psi_i)\in \Q_2$, it follows that $\det D(\psi_i)=1$.

    It remains to construct a Galois pair for $D_{pc}(4,\pm 2,i,\pm)$. For $a_2=-2$, we set $E_0=E_1$. The automorphism $u(x,y)=(x+1,y+x)$ generates $\Aut_{\F_2}(E)=C_4$ and so it satisfies $u^2=[-1]$. We then choose $\Gamma=C_4$ and $\nu(x)=u^{\pm 1}$. Since $K$ is totally ramified, it follows directly that $\varphi$ satisfies $X^2-a_2X+2$. From $u^2=[-1]$ we deduce $D(u)^2=-\Id$ thus the characteristic polynomial of $D(u)$ is $X^2+1$ so $\det D(u)=1$. Here it is clear that $\nu$ is injective and $\F_2$ is minimal. 

 Thus, by Proposition~\ref{P: classification C4}, these Galois pairs must be either $D_{pc}(4,-2,i,+)$ or $D_{pc}(4,-2,i,-)$ with one corresponding to $\nu(x)=u$ and the other to $\nu(x)=u^{-1}$. For $a_2=2$, we can use the same argument with $E_0=E_5$ together with the same $\Gamma$, $u$ and $\nu$ as for $a_2=-2$.
 
\end{proof}

\subsection{Proof of the converse}

We first introduce the following construction due to Volkov, see \cite[4.3]{Volkov2001} for a detailed exposition of it.
    
Let $K/\Q_2$ be a finite Galois extension, $\cl G/\cl O_K$ be a $2$-divisible group, $G=\cl G\otimes_{\cl O_K}k$ and $T_2(\cl G)$ its Tate module. Suppose that the natural action of $G_K$ on $T_2(\cl G)$ can be extended to a $G_{\Q_2}$-action $$\rho:G_{\Q_2}\rightarrow Aut_{\Z_2}(T_2(\cl G)).$$ This extended action induces a semilinear action of $\Gal(K/\Q_2)$ on $D(G)$ and on $D_{cris,K}^{*}(V_2(\cl G))$ as follows:

Let $g\in \Gal(K/\Q_2)$ and let $\tilde{g}\in G_{\Q_2}$ be a lift. The action of $\tilde{g}$ on the $\overline{\Q}_2$-points of $\cl G$ induces a linear map $\tilde{g}:T_2(\cl G)\rightarrow T_2(\cl G^{g})$. Then, the map $\tau_g:=\tilde{g}\circ \rho(\tilde{g})^{-1}$ is a $G_K$-equivariant map independent of the choice of lift $\tilde{g}$. By the full faithfulness of the Tate-module functor, there exists a unique isomorphism $f_{g}:\cl G\rightarrow \cl G^{g}$ such that $T_2(f_{g})=\tau_g$. These maps satisfy $f_{gh}=f_{h}^g\circ f_g$.

Applying $D_{cris,K}^{*}(V_2(-))$ to this map gives $D_{cris,K}^{*}(V_2(f_g)):D_{cris,K}^{*}(V_2(\cl G^g))\rightarrow D_{cris,K}^{*}(V_2(\cl G))$. Composing this map with the isomorphism $\iota^{cris}_g:D_{cris,K}^{*}(V_2(\cl G))^g\rightarrow D_{cris,K}^{*}(V_2(\cl G^g))$ and with the semilinear map $m^{cris}_g:D_{cris,K}^{*}(V_2(\cl G))\rightarrow D_{cris,K}^{*}(V_2(\cl G))^g$ gives the semilinear map $$b_g:D_{cris,K}^{*}(V_2(\cl G))\xrightarrow{m_g^{cris}}D_{cris,K}^{*}(V_2(\cl G))^g\xrightarrow{\iota_g^{cris}}D_{cris,K}^{*}(V_2(\cl G^g))\xrightarrow{D_{cris,K}^{*}(V_2(f_g))}D_{cris,K}^{*}(V_2(\cl G)).$$ The map $g\mapsto b_g$ gives the desired semilinear action of $\Gal(K/\Q_2)$ on $D_{cris,K}^{*}(V_2(\cl G))$. Moreover, this semilinear action coincides with the action on $D_{cris,K}^{*}(V_2(\cl G))$ viewed as the $\fpgmt$ corresponding to the potentially crystalline extended representation $\rho:G_{\Q_2}\rightarrow Aut_{\Q_2}V_2(\cl G)$.

For $D(G)$, the reduction of $f_{g}$ gives a map $\overline{f}_g:G\rightarrow G^{\bar{g}}$, where $\bar{g}$ is the class of $g$ in $\Gal(k/\F_2)$. Then the rational contravariant Dieudonn\'e functor gives a $K_0$-linear map $D(\bar{f}_g):D(G^{\bar{g}})\rightarrow D(G)$. Now we can compose $D(\bar{f}_g)$ with the canonical $K_0$-linear isomorphism $\iota^D_g:D(G)^{g}\rightarrow D(G^{\bar{g}})$ and the $g$-semilinear map $m^D_g:D(G)\rightarrow D(G)^g$ to obtain a $g$-semilinear map: $$c_g:D(G)\xrightarrow{m^D_g}D(G)^g\xrightarrow{\iota^D_g}D(G^{\bar{g}})\xrightarrow{D(\bar{f}_g)}D(G)$$ The map $g\mapsto c_g$ gives the desired semilinear action of $\Gal(K/\Q_2)$ on $D(G)$.

Fontaine in~\cite[Theorem 6.2]{Fontaine1982} shows that, for a $p$-divisible group $\cl G$, if we equip $D(G)\otimes_{K_0}K$ with the Hodge filtration determined by the lift $\cl G/\cl O_K$, then there is a natural isomorphism $\gamma_{\cl G}:D(G)\rightarrow D_{cris,K}^{*}(V_2(\cl G))$ of filtered $\varphi$-modules. This isomorphism is natural and compatible with base change. By naturality, $$D_{cris,K}^{*}(V_2(f_{g}))\circ \gamma_{\cl G^{g}}=\gamma_{\cl G}\circ D(\bar{f}_{g})$$ and by compatibility with base change, $$(\iota^{cris}_g\circ m^{cris}_g)\circ \gamma_{\cl G}=\gamma_{\cl G^{g}}\circ(\iota^D_g\circ m^D_g)$$ Combining these identities, we see that $b_g\gamma_{\cl G}=\gamma_{\cl G}c_g$, that is, with the actions we have just defined Fontaine comparison isomorphism is an isomorphism of $\fpgmt$. This fact will play a key role in the proof of Theorem~\ref{T: explicit realization}.

The last ingredient we need to construct our curves is the following descent criterion due to Volkov.

\begin{lemma}\label{L: Volkov criterion}
    Let $L/\Q_2$ be one of the fields in Table~\ref{Ta: InTypes} and $K$ its Galois closure. Let $E/L$ be an elliptic curve with good reduction and let $\tilde{E}$ be the reduction of $E\otimes_L K$. Then $E$ can be defined over $\Q_2$, i.e., $E=E_0\otimes_{\Q_2} L$ for some $E_0/\Q_2$, if and only if the $G_{K}$-action on $T_{2}(E)$ can be extended to a $G_{\Q_2}$-action such that:

    \begin{itemize}
        \item The action of $G_{L}$ on $T_2(E)$ coincides with the usual one.
        \item The action of $G_{K_0}$ on $T_2(E)$ induces, via the $c_g$ defined above, an action of $I_{K/\Q_2}$ on $D(\tilde{E})$ that comes from $k$-automorphisms of $\tilde{E}$.
    \end{itemize}
    In this case, the extended $G_{\Q_2}$-action on $T_2(E)$ is isomorphic to $T_2(E_0)$.
  
\end{lemma}
\begin{proof}
    See \cite[Subsection 1.2.2]{Volkov2005}.
\end{proof}

The proof of the following theorem closely follows that of~\cite[Theorem~5.7]{Volkov2005}, with results requiring $p$ to be odd replaced by their counterparts for $p=2$.

\begin{theorem}\label{T: explicit realization}
    Let $D$ be one of the $\fpgmt$s in Table~\ref{Ta: Main}. Then there exists an elliptic curve $E/\Q_2$ such that $D_{cris,K}^{*}(V_2(E))\simeq D$.
 
\end{theorem} 
\begin{proof}
    By Remark~\ref{R: reduce to 3,4,8,24}, it suffices to prove the result for $e\in\{3,4,8,24\}$. Thus, assume that $D$ is one of the modules corresponding to these values of the semistability defect. In the cases $e=8$ and $e=24$, we know that there is a Galois pair for at least one of the two signs; in these cases, let $D$ be a module with such a sign. 

    By the antiequivalence of categories between Galois representations over $\Q_2$ that become crystalline over $K$ and $\fpgmt$s, there exists a $G_{\Q_2}$-representation $V$ with an isomorphism $i_{cris}:D_{cris,K}^{*}(V)\xrightarrow{\sim} D$.
    
     Recall that for an elliptic curve we defined $D(E):=\ma D(E[2^\infty])[1/2]$. By Lemma~\ref{L: Galois pairs exist}, there exists a Galois pair $(E_0,\Gamma,\nu)$ with an isomorphism $i_E:D(E)\xrightarrow{\sim} D$ of unfiltered $(\varphi,\Gal(K/\Q_2))$-modules. We denote the action of $g\in\Gal(K/\Q_2)$ on $D(E)$ by $a_g$. Taking $\Gal(K/L)$-invariants gives $i_{E_0}:D(E_0)\xrightarrow{\sim}D_0$, where $D_0=D^{\Gal(K/L)}=D^{\langle t\rangle}$.

    Recall that $L/\Q_2$ is one of the totally ramified extensions in Table~\ref{Ta: InTypes} and $K$ its Galois closure. Let $T$ be a $G_{\Q_2}$-stable lattice in $V$. Since $K/L$ is unramified and $V\vert_{G_K}$ is crystalline, $V\vert_{G_L}$ is crystalline, so $T\vert_{G_L}$ is crystalline. Since $T$ also has Hodge--Tate weights $(0,1)$, it follows from \cite[Theorem~2.2.1]{Liu2013} that there exists a $2$-divisible group $\cl G_0/\cl O_L$ and an isomorphism $\theta:T\vert_{G_L} \xrightarrow{\sim}T_2(\cl G_0)$ of $G_L$-representations. Let $G_0=\cl G_0\otimes_{\cl O_L}\F_2$, $\cl G=\cl G_0\otimes_{\cl O_L}\cl O_K$ and $G=G_0\otimes_{\F_2}k$. Via $\theta$, we can transport the $G_{\Q_2}$-action on $T$ to $T_2(\cl G)$. By the construction described at the beginning of \S11.3, we can equip $D_{cris,K}^{*}(V_2(\cl G))$ and $D(G)$ with actions $b_g$ and $c_g$, respectively, so that Fontaine's comparison isomorphism $\gamma_{\cl G}:D(G)\xrightarrow{\sim} D_{cris,K}^{*}(V_2(\cl G))$ is $\Gal(K/\Q_2)$-equivariant and $D_{cris,K}^{*}(\theta):D_{cris,K}^{*}(V_2(\cl G_0))\xrightarrow{\sim}D_{cris,K}^{*}(V)$ is an isomorphism. In this way, we obtain a $\Gal(K/\Q_2)$-equivariant isomorphism: $$i_{G}:D(G)\xrightarrow{\gamma_{\cl G}}D_{cris,K}^{*}(V_2(\cl G))\xrightarrow{D_{cris,K}^{*}(\theta)}D_{cris,K}^{*}(V)\xrightarrow{i_{cris}}D$$ As for $i_{E}$, taking $\Gal(K/L)$-invariants gives a map $i_{G_0}:D(G_0)\rightarrow D_0$.
    
    Let $j_0=i_{E_0}^{-1}\circ i_{G_0}$ and $j=i_{E}^{-1}\circ i_{G}$ be its extension of scalars to $K_0$. The Galois equivariance of $j$ gives $$j \circ c_g=a_g\circ j.$$

    For $g\in I_{K/\Q_2}$, the residue action is trivial, so $a_g=D(\nu(g))$ and $c_g=D(\overline f_g)$. By rational Dieudonn\'e theory, $j_0$ corresponds contravariantly to a quasi-isogeny $r_0:E_0[2^\infty]\dashrightarrow G_0.$ Let $r=r_0\otimes_{\F_2}k$, so that $D(r)=j$. Choose $n\geq0$ so that $q_0=2^nr_0$ is an isogeny, and put $q=q_0\otimes k=2^nr$. The equivariance identity gives $\overline f_g\circ r=r\circ\nu(g)$ in the quasi-isogeny category. Since homomorphism groups of $2$-divisible groups have no $2$-torsion, this implies the equality $\overline f_g\circ q=q\circ\nu(g).$
Thus $H=\ker(q)$ is stable under $\nu(g)$ for every $g\in I_{K/\Q_2}$. Put $H_0=\ker(q_0)$; then $H=H_0\otimes_{\F_2}k$. The group scheme $H_0$ is contained in $E_0[2^m]$ for some $m$, so it is a finite subgroup scheme of the elliptic curve $E_0$.

    If we consider the elliptic curve $E_0'=E_0/H_0$ together with the quotient map $\pi_0:E_0\rightarrow E'_0$, it follows that $2^nr_0$ factors through $\pi_0$ as $2^nr_0=\lambda_0\circ \pi_0$, where $\lambda_0:E_0'[2^{\infty}]\rightarrow G_0$ is an isomorphism. Let $E^{\prime}=E_0^{\prime}\otimes_{\F_2}k$ and $\pi=\pi_0\otimes_{\F_2}k$. Since $\nu(g)(H)=H$, it follows that for each $\nu(g)$ there exists a unique $\nu'(g):E'\rightarrow E'$ such that $\pi\circ\nu(g)=\nu'(g)\circ \pi$. Since $\nu$ is an antihomomorphism, the uniqueness of $\nu'(g)$ implies that $\nu':I_{K/\Q_2}\rightarrow Aut_k(E')$ is an antihomomorphism.

    We now use the Serre--Tate functor to produce the desired elliptic curve. For each $n$, let $R_n=\cl O_L/\p^n$. Applying the Serre--Tate functor see~\cite[Theorem 1.2.1]{KatzSerreTate1981} to $(E'_0,\cl G_0\otimes R_n,\lambda_0)$ produces a compatible system of elliptic schemes $\cl E_n/R_n$ with special fiber $E'_0$ together with isomorphisms $\lambda_n:\cl E_n[2^{\infty}]\rightarrow \cl G_0\otimes_{\cl O_L}R_n$ that lift $\lambda_0$. This provides a formal scheme $\hat{\cl E}$ which, by Grothendieck's criterion for algebraization \cite[Theorem~5.4.5]{EGAIII1} applied to the compatible system of ample sheaves $\cl O(3[0]_n)$, gives an elliptic curve $\tilde{E}_0/\cl O_L$ whose reduction is $E_0'$ together with an isomorphism $\Lambda_0:\tilde{E}_0[2^{\infty}]\rightarrow \cl G_0$ whose reduction is $\lambda_0$. Let $\tilde{E}=\tilde{E}_0\otimes_{\cl O_L}\cl O_K$ and $\Lambda:\tilde{E}[2^{\infty}]\rightarrow \cl G$.

    To finish the proof, we use Lemma~\ref{L: Volkov criterion}. Using $\theta^{-1}\circ T(\Lambda): T_2(\tilde{E})\rightarrow T$, we can extend the $G_K$-action on $T_2(\tilde{E})$ to a $G_{\Q_2}$-action. Since both $T(\Lambda)$ and $\theta$ are $G_L$-equivariant, it follows that $T_2(\tilde{E})$ with its natural $G_L$-action is isomorphic to the extended action induced by $\theta^{-1}\circ T(\Lambda)$. It only remains to show that the $I_{K/\Q_2}$-action induced on $D(E')$ by the extended $G_{\Q_2}$-action $\rho_{\tilde{E}}$ is given by automorphisms of $E'$. Since the $G_{\Q_2}$-action on $T_2(\tilde{E})$ is defined through $\theta^{-1}\circ T_2(\Lambda)$ and the action on $T_2(\cl G)$ is defined through $\theta$, $T_2(\Lambda)$ must be a $G_{\Q_2}$-isomorphism. Therefore, the following map is $\Gal(K/\Q_2)$-equivariant: $$\Phi:D(E')
\xrightarrow{\gamma_{\tilde E}}
D_{cris,K}^{*}(V_2(\tilde E))
\xrightarrow{D_{cris,K}^{*}(V_2(\Lambda))^{-1}}
D_{cris,K}^{*}(V_2(\cl G))
\xrightarrow{\gamma_{\cl G}^{-1}}
D(G)
\xrightarrow{j}
D(E)$$

Let $\lambda=\lambda_0\otimes_{\F_2}k$ be the reduction of
$\Lambda$. By the naturality of Fontaine's
comparison isomorphism, we have $D_{cris,K}^{*}(V_2(\Lambda))\circ\gamma_{\cl G}=\gamma_{\tilde E}\circ D(\lambda)$, so
$$
\Phi=j\circ\gamma_{\cl G}^{-1}
       \circ D_{cris,K}^{*}(V_2(\Lambda))^{-1}
       \circ\gamma_{\tilde E}
  =j\circ D(\lambda)^{-1}.
$$

Moreover, the equality $2^nr=\lambda\circ\pi$ gives $2^nj=D(2^nr)=D(\pi)\circ D(\lambda)$. It follows that $\Phi=2^{-n}D(\pi)$ and,
since $\Phi$ is $\Gal(K/\Q_2)$-equivariant, so is $D(\pi)$. If $c'_g$ denotes the action on $D(E')$ induced by
$\rho_{\tilde E}$, the
equivariance of $\pi$ for $g \in I_{K/\Q_2}$ reads $D(\pi)\circ c'_g=a_g\circ D(\pi)$. Applying the contravariant Dieudonn\'e functor to
$\nu'(g)\circ\pi=\pi\circ\nu(g)$ gives $D(\pi)\circ D(\nu'(g))=D(\nu(g))\circ D(\pi)=a_g\circ D(\pi)$.
Since $\pi$ is an isogeny, $D(\pi)$ is an isomorphism of rational Dieudonn\'e modules and so $c_g'=D(\nu'(g))$. Thus the induced inertia action on $D(E')$ comes from
$k$-automorphisms of $E'$, as required. Hence, by applying Lemma~\ref{L: Volkov criterion}, we deduce that the generic fiber of $\tilde{E}$ descends to an elliptic curve defined over $\Q_2$ whose rational Tate module is precisely $V$, as desired.

This proves the existence of elliptic curves whose $D_{pcris}^{*}(V_2(E))$ realizes each of the $\fpgmt$s with $e\in\{3,4\}$ and each of the modules with one of the signs for $e\in\{8,24\}$. By Lemmas~\ref{L: Twist for e=8} and~\ref{L: Twist for e=24}, twisting induces a bijection on the admissible filtrations, so we obtain curves for all modules with $e\in\{8,24\}$. This finishes the proof.
\end{proof}

\end{document}